\documentclass[11pt]{amsart}

\usepackage[english]{babel}
\usepackage{hyperref}
\usepackage{amsmath}
\usepackage{amsfonts}
\usepackage{amssymb}
\usepackage{amsthm}
\usepackage{thmtools}
\usepackage{color}
\usepackage{tikz}
\usepackage{tikz-cd}
\usepackage{tikz-3dplot}
\usepackage{enumitem}
\usetikzlibrary{babel, intersections, decorations.markings, decorations.pathreplacing}

\allowdisplaybreaks

\definecolor{dunkelrot}{RGB}{210,30,30}

\setlist[enumerate,1]{label={\arabic*.}}

\declaretheorem[name=Theorem, style=plain, numberwithin=section]{satz}
\declaretheorem[name=Lemma, style=plain, sibling=satz]{lem}
\declaretheorem[name=Corollary, style=plain, sibling=satz]{kor}

\declaretheorem[name=Theorem, style=plain, numbered=no]{satzA}

\declaretheorem[name=Definition, style=definition, qed=$\blacktriangle$, sibling=satz]{defi}

\declaretheorem[name=Remark, style=definition, qed=$\blacklozenge$, sibling=satz]{bem}

\newcommand{\N}{\mathbb{N}}
\newcommand{\Z}{\mathbb{Z}}

\newcommand{\R}{\mathbb{R}}

\renewcommand{\S}{\mathbb{S}}

\newcommand{\Vol}{\operatorname{Vol}}

\renewcommand{\epsilon}{\varepsilon}

\newcommand{\Isom}{\operatorname{Isom}}
\newcommand{\id}{\operatorname{id}}

\newcommand{\tors}{\operatorname{tors}}

\newcommand{\K}{\mathbb{K}}
\renewcommand{\H}{\mathbb{H}}

\newcommand{\FIX}{\operatorname{Fix}}

\newcommand{\arsinh}{\operatorname{arsinh}}
\newcommand{\artanh}{\operatorname{artanh}}

\begin{document}

\title{Homology bounds for hyperbolic orbifolds}

\author{Hartwig Senska}
\address{Karlsruhe Institute of Technology, Germany}
\email{hartwig.senska@alumni.kit.edu}

\begin{abstract}
We will provide bounds on both the Betti numbers and the torsion part of the homology of hyperbolic orbifolds. These bounds are linear in the volume and are a direct consequence of an efficient simplicial model of the thick part, which we will construct as well. The homology statements complement previous work of Bader, Gelander and Sauer (torsion homology of manifolds), Samet (Betti numbers of orbifolds) and a classical theorem of Gromov (Betti numbers of manifolds).

For arithmetic, non-compact hyperbolic orbifolds -- i.e. in the case of arithmetic, non-uniform lattices in $\Isom(\H^n)$ -- the strongest results will be obtained.
\end{abstract}
\maketitle

\section{Introduction}

An interesting feature of negative curvature is that the topology (e.g. in terms of the homology) of sufficiently well-behaved spaces like manifolds or orbifolds can in some sense be controlled by their volume. The following result of Gromov \cite{BGS} can be seen as a starting point for studies into this theme: for a Hadamard $n$-manifold $X$ with pinched negative sectional curvature\footnote{In the analytic case, this extends to non-positive curvature $-1 \leq K \leq 0$, given that there are no Euclidean de Rham factors in $X$.} $-1 \leq K \leq a < 0$ (for some $a<0$) and a torsion-free lattice $\Gamma < \Isom(X)$, the Betti numbers -- that is, the free part of the homology -- of the quotient manifold $X\slash\Gamma$ is linearly bounded by the volume, i.e.
\[
b_k(X\slash\Gamma; \K) \leq C \cdot \Vol(X\slash\Gamma).
\]
This holds for all degrees $k=0,\ldots,n$ and arbitrary coefficient fields $\K$. Here, $C = C(n) > 0$ is a constant depending only on the dimension $n$.

Using a similar Morse theoretic argument to the one of Gromov, Samet \cite{Samet} showed that this can be extended to general lattices $\Gamma < \Isom(X)$, i.e. where $X\slash\Gamma$ is an orbifold and not necessarily a manifold. Again, the Betti numbers of the quotient orbifold satisfy
\[
b_k(X\slash\Gamma; \K) \leq C \cdot \Vol(X\slash\Gamma),
\]
where $k=0,\ldots,n$; here, the coefficient field $\K$ has to have characteristic $0$. Moreover, $C = C(n, \eta) > 0$ will no longer depend only on the dimension $n$, but also on an upper bound $\eta$ on the order of finite subgroups of $\Gamma$ (which always exists), i.e. $|G| \leq \eta$ for all $G < \Gamma$ finite\footnote{This extra dependency on $\eta$ is not stated explicitly in the paper \cite{Samet}, but it can be seen that the proof indeed uses $\eta$. Moreover, there are geometric reasons that strongly suggest that it won't be possible to dispose of this additional assumption.}.

As the Betti numbers constitute only one part of the homology, it is a natural question whether the other part -- the torsion in the homology -- might admit similar bounds. Bader, Gelander and Sauer \cite{Sauer} settled this question positively for the case of negatively pinched manifolds (i.e. torsion-free latices again), showing that
\[
\log |\tors H_k(X\slash\Gamma; \Z)| \leq C \cdot \Vol(X\slash\Gamma)
\]
holds for all $k=0,\ldots,n$, where $C = C(n) > 0$ is a constant depending only on the dimension $n$. As a special case, the statement for degree $k=1$ in dimension $n=3$ has to be exlcuded; using Dehn surgery, \cite{Sauer} gives an explicit counterexample for that situation. Unlike Gromov \cite{BGS} and Samet \cite{Samet}, Bader, Gelander and Sauer \cite{Sauer} don't employ Morse theory to prove the statement above; instead, they construct an efficient simplical decomposition of the thick part of $X\slash\Gamma$, which yields the torsion homology result as a direct consequence. This decomposition would also imply another proof for Gromov's theorem above. The curvature conditions in \cite{Sauer} can be relaxed to negatively curved visibility manifolds, as was shown by the author of the present paper in \cite{SenskaVis}.

We will provide answers for the remaining case of torsion homology for orbifolds (i.e. general, not necessarily torsion-free lattices) in the hyperbolic setting. Similar to \cite{Sauer}, we will achieve this by first constructing an efficient simplical decomposition of the thick part of the orbifold. To fix notation, $\Gamma < \Isom(\H^n)$ will denote a lattice, $\eta \in \N$ an upper bound on the order of finite subgroups of $\Gamma$ and $\nu > 0$ a lower bound on the displacement of hyperbolic elements\footnote{In the manifold setting, this would be equivalent to a lower bound on the length of geodesic loops in $M$; the interpretation is only slightly more complicated in our orbifold situation.} of $\Gamma$ (both $\eta$ and $\nu$ always exist). We let $M := \H^n\slash\Gamma$ denote the quotient orbifold and write $M_+$ for its thick part. Our main result now states:

\begin{satzA}[see Theorem \ref{Satz:Hauptresultat_beschraenkt_komplizierter_Simplizialkomplex}]
There are constants $C = C(n,\eta,\nu) > 0$ and $D = D(n,\nu) > 0$, such that for any such orbifold $M$, the pair $(M_+, \partial M_+)$ -- i.e. the thick part and its boundary -- is as a pair homotopy equivalent to a simplicial pair $(S, S')$, where the number of vertices of $S$ is bounded by $C \cdot \Vol(M)$ and the degree at the vertices of $S$ is universally bounded by $D$. For arithmetic, non-uniform $\Gamma$, the constants $C$ and $D$ will only depend on the dimension $n$.
\end{satzA}

By a straightforward Mayer-Vietoris argument, our simplicial decomposition yields another proof for the linear bounds on the Betti numbers as a byproduct.

\begin{satzA}[see Theorem \ref{Satz:Freie_Homologie_Schranke}]
There is a constant $E = E(n, \eta, \nu) > 0$ such that for any such orbifold $M$, we have
\[
b_k(M;\K) \leq E \cdot \Vol(M)
\]
for all $k \in \N_0$ and arbitrary coefficient field $\K$. For arithmetic, non-uniform $\Gamma$, the constant $E$ will only depend on the dimension $n$.
\end{satzA}

While Samet's \cite{Samet} result above already provided a linear bound for the Betti numbers even under more general curvature assumptions, it is restricted to coefficient fields of characteristic $0$; our statement has the advantage that it is valid for arbitrary coefficients.

The application of central interest is the torsion of the homology, where we will show:

\begin{satzA}[see Theorem \ref{Satz:Torsion_Homologie_Schranke}]
There is a constant $F = F(n, \eta, \nu) > 0$ such that for any such orbifold $M$, we have
\[
\log | \tors H_k(M;\Z) | \leq F \cdot \Vol(M)
\]
for all $k \in \N_0$. For arithmetic, non-uniform $\Gamma$, the constant $F$ will only depend on the dimension $n$.
\end{satzA}

Note that we did not exclude the case of degree $k=1$ in dimension $n=3$, in contrast to the similar statement of Bader, Gelander and Sauer \cite{Sauer} above; our extra restrictions regarding $\nu$ enable us to do so.

Although the general situation of our statements uses assumptions on $\eta$ and $\nu$, the results are as good as possible for the arithmetic, non-uniform case. Recall that arithmetic, non-uniform lattices in $\Isom(\H^n)$ correspond to arithmetic, non-compact hyperbolic orbifolds. They form an interesting and widely studied class of orbifolds, with the maybe most prominent examples given by the Bianchi orbifolds (in dimension $n=3$).

Our restriction to the hyperbolic setting is mainly for technical reasons, as we will use convexity arguments that only hold in constant curvature to prove that the cover we construct for the thick part is indeed a good cover; the goodness of the cover is needed to be able to use (some appropriate modification of) the Nerve lemma to obtain the desired simplicial model.

We strongly suspect that it will not be possible to eliminate $\eta$ from the assumptions. This is mainly due to a geometric reason, which also manifests itself in the following, simplified example: for orbifolds, a tube in the thin part takes the form\footnote{Tubes might also be of other forms than the one stated here (see our thick-thin decomposition Theorem \ref{Satz:Dick_duenn_Zerlegung_Orbifaltigkeit}), but this does not invalidate our general argument.} of a bundle over $\S^1$, with the fiber being the quotient of an $(n-1)$-disk by some finite subgroup $E$ of $\Gamma$. The volume of such a an orbifold tube is the quotient of the volume of the corresponding tube with trivial $E$ (of which we can think of as a manifold tube) by the order of $E$. In other words, $|E|$-many such orbifold tubes would contribute the same amount of volume as one corresponding manifold tube. If with rising volume\footnote{Note that due to Wang's theorem, it will not be possible to check this idea by just looking at lattices of a given volume; we need to increase the volume of the admissible lattices to obtain enough new orbifolds to look at.}, there was no joint bound $\eta$ on the order of finite subgroups of the lattices, disproportionally many ever smaller tubes (in terms of volume) could be fitted into the corresponding orbifolds; e.g. it could be possible that with linear growth of the volume, the number of tubes would grow polynomially or exponentially. Since the tubes are homotopic to $\S^1$, they could each contribute another free summand to the first homology, or -- in other terms -- increase the first Betti number by $1$; hence, the first Betti number would grow in orders larger than the volume. Thus to be able to bound the Betti numbers linearly by the volume, it seems necessary to assume a joint bound $\eta$ on the order of finite subgroups of the lattices.

The situation is different for $\nu$ -- which was needed for technical reasons only -- and we suspect that it could be eliminated from all the above statements. In fact, more general versions of the results presented here (see \cite{SenskaDissertation}) already no longer rely on $\nu$; but as of now, this comes at the price of having bounds polynomial in the volume\footnote{Essentially in the form of $C \cdot \Vol(M)^{k+1}$ for a constant $C = C(n, \eta) > 0$ and homology degree $k$.} instead of the linear ones presented in this paper. As these slight generalizations need significantly more technical work, we refrain from presenting them here.

\subsection{Structure of the paper}

The following section \ref{Kapitel:Grundlagen} summarizes some well-known facts utilized throughout this paper. As a next step, in section \ref{Kapitel:Dick-duenn-Zerlegung} we state the thick-thin decomposition for orbifolds in a fairly general setting. In section \ref{Kapitel:Hauptresultat}, we restrict ourselves to the hyperbolic case and prove the main result, namely the efficient simplicial model for the thick part of hyperbolic orbifolds. The main applications for the homology of such orbifolds are contained in the final section \ref{Kapitel:Anwendungen}.

\subsection{Acknowledgement}

The results presented here are part of my doctoral thesis \cite{SenskaDissertation}, which was written under the supervision of Prof. Roman Sauer, to whom I am grateful for supporting my work. The author acknowledges funding by the Deutsche Forschungsgemeinschaft (DFG, German Research Foundation) -- 281869850 (RTG 2229).

\section{Preliminaries}\label{Kapitel:Grundlagen}

First, let us fix notation and state some useful facts. General references for the concepts covered here are \cite{BGS} and \cite{BenPet}.

We will always let $X$ denote a Hadamard manifold with curvature $-1 \leq K \leq a < 0$ for some $a<0$. For a lattice $\Gamma < \Isom(X)$, $M := X\slash\Gamma$ will be the finite-volume orbifold to be studied ($M$ is a manifold if and only if $\Gamma$ is torsion-free). The boundary at infinity of $X$ will be denoted by $X(\infty)$ or sometimes $\partial X$.

Every isometry $\gamma \in \Isom(X)$ gives rise to a displacement function
\[
d_{\gamma}: X \rightarrow [0,\infty), \qquad x \mapsto d_{\gamma}(x) := d(x, \gamma x),
\]
which can be used to classify the nontrivial isometries of $\Isom(X)$: $\gamma$ is elliptic if $d_{\gamma}$ has minimum $0$; it is hyperbolic if $d_{\gamma}$ has minimum $>0$; and it is parabolic if $d_{\gamma}$ has no minimum. Elliptic isometries are precisely the torsion elements of $\Gamma$; if $M$ is non-compact, there has to be a parabolic $\gamma \in \Gamma$. The different isometry types are stable under taking powers (with powers $\neq 0$). Hyperbolic isometries have precisely two fixed points in $X(\infty)$, whereas parabolic isometries have precisely one.

Since the distance function of $X$ is convex, the same holds for the displacement functions. Hence the sublevel sets $\{ d_{\gamma} < a \}$ and $\{ d_{\gamma} \leq a \}$ (for $a \geq 0$) are convex as well. Fixed point sets $\FIX(\gamma)$ of elliptic isometries $\gamma$ arise in the special case of $a=0$; they are complete, totally geodesic submanifolds of codimension $\geq 1$ (sometimes we treat $\id$ as an elliptic isometry, in which case $\FIX(\id) = X$ has codimension $0$).

For a closed convex set $W \subseteq X$, there is a well-defined projection $\pi_W: X \rightarrow W$ sending a point $x \in X$ to the (unique) point $\pi_W(x) \in W$ of smallest distance to $x$; we will call $\pi_W(x)$ the projection point or foot point of $x$ in $W$. This projection is equivariant under isometries preserving $W$, i.e. if $\gamma \in \Isom(X)$ with $\gamma W = W$, then $\pi_W(\gamma x) = \gamma \pi_W(x)$ for all $x \in X$.

We will adopt some notation from \cite{Samet} regarding singular submanifolds. For $G < \Isom(X)$, let $F(G) := \bigcap_{g\in G} \FIX(g)$. Now for $\Delta < \Isom(X)$, define
\begin{align*}
\Sigma(\Delta) &:= \{ \FIX(G) : G < \Delta \text{ finite} \}, \\
\Sigma_i(\Delta) &:= \{ Y \in \Sigma(\Delta) : \dim(Y) = i \}, \text{ and}\\
\Sigma_{<i}(\Delta) &:= \{ Y \in \Sigma(\Delta) : \dim(Y) < i \}.
\end{align*}
Moreover,
\[
S_i(\Delta) := \bigcup_{Y\in \Sigma_i(\Delta)} Y \qquad\text{and}\qquad S_{<i}(\Delta) := \bigcup_{Y\in \Sigma_{<i}(\Delta)} Y.
\]
We will often omit $\Delta$ in the notation if it is obvious from the context. Note that since $\id \in \Delta$ (for arbitrary $\Delta < \Isom(X)$), we always have $X = \FIX(\{\id\}) \in \Sigma(\Delta)$.

To prove our main results, we will restrict ourselves to the hyperbolic space $X = \H^n$ and use the upper half space model:
\[
\H^n = \{ (x, t) = (x_1, \ldots, x_{n-1}, t) \in \R^{n-1} \times \R \,|\, t > 0 \} = \R^{n-1} \times \R_{>0} \subseteq \R^n
\]
with the usual hyperbolic metric. In $\H^n$, a nonempty, closed subset is convex if and only if it is the intersection of all its (closed) supporting half spaces (\cite{CEM} Proposition II.1.4.1). For the upper half space model, the distances are given by the following formulas:
\begin{itemize}
\item $d((x,t), (y,s)) = 2 \cdot \artanh \left( \sqrt{\frac{\|x-y\|^2 + (t-s)^2}{\|x-y\|^2 + (t+s)^2}} \right)$,

\item $d((x,t), (x,s)) = \left| \ln\left( \frac{t}{s} \right) \right|$,

\item $d((x,t), (y,t)) = 2 \cdot \arsinh \left( \frac{\|x-y\|}{2 \cdot t} \right)$.
\end{itemize}
Here, $\|\cdot\|$ denotes the usual Euclidean norm on the $\R^{n-1}$-factor.

\section{Thick-thin decomposition}\label{Kapitel:Dick-duenn-Zerlegung}

The general idea behind the thick-thin decomposition is most obvious in the manifold case: the manifold can be decomposed into two parts -- thick part and thin part --, with the thick part being characterized by a uniform lower bound on the injectivity radius; this makes its geometry easy to control. On the other hand, the thin part turns out to consist of only two types of components, tubes (which are ball bundles over the circle) and cusps (which are products of a compact manifold with a ray). Tubes correspond to sublevel sets of (the displacement function of) hyperbolic isometries in the universal cover, while cusps similarly correspond to sublevel sets of parabolic isometries. This relies on the absence of elliptic isometries in the lattice $\Gamma$. While the general situation is similar in the orbifold case, the elliptic isometries occurring now might complicate the picture. As we will see, using our slightly different construction, we can essentially remove the influence elliptic isometries might have: the thin part will be given by the sublevel sets of parabolic and hyperbolic isometries only.

Arguably the most essential tool in the thick-thin decomposition is the Margulis lemma:

\begin{satz}[Margulis lemma; \cite{Samet} Theorem 2.1]\label{Samet2.1}
There are constants $\epsilon(n) > 0$ and $m(n) \in \N$ depending only on $n$, such that if $X$ is an $n$-dimensional Hadamard manifold with sectional curvature $-1 \leq K \leq 0$, then for every discrete group $\Gamma < \Isom(X)$, every $x \in X$ and every $\epsilon \leq \epsilon(n)$, the group
\[
\Gamma_{\epsilon}(x) := \langle \{ \gamma \in \Gamma : d_{\gamma}(x) < \epsilon \} \rangle
\]
contains a nilpotent normal subgroup $N$ of index $\leq m(n)$. If $\Gamma_{\epsilon}(x)$ is finite, then $N$ is abelian.
\end{satz}

The constants $\epsilon(n)$ and $m(n)$ in Theorem \ref{Samet2.1} will be called Margulis $\epsilon$ and Margulis index constant, respectively.

In the standard thick-thin decompositions (see e.g. \cite{BGS} chapter 10, \cite{BenPet} chapter D; moreover, \cite{Bowditch} chapter 3.5 is interesting as it explicitly covers the orbifold case), all isometries are treated equally: the sublevel sets in question are given by $\{ d_{\gamma} < \epsilon \}$ for some fixed $\epsilon \in (0,\epsilon(n)]$. \cite{Sauer} already introduced the concept of varying levels $\epsilon_{\gamma}$ for different $\gamma$ for manifolds (see also \cite{SenskaVis} chapter 3 for a proof). Our definition extends this to the orbifold case.

Let $\epsilon \in (0, \epsilon(n)/2]$ be arbitrary, but fixed\footnote{We will generally choose $\epsilon$ to be some fixed fraction of $\epsilon(n)$ to make sure all constants only depend on the dimension $n$.}, and let $\Gamma \ni \gamma \mapsto \epsilon_{\gamma} \in [\epsilon, \epsilon(n)/2]$ be a conjugation-invariant choice of levels, i.e. $\epsilon_{\gamma \gamma' \gamma^{-1}} = \epsilon_{\gamma'}$ for all $\gamma, \gamma' \in \Gamma$. Define
\[
\Gamma_{\epsilon_{\Gamma}}(x) := \langle \gamma \in \Gamma : d_{\gamma}(x) < \epsilon_{\gamma} \rangle
\]
and consequently the \textbf{thick part} $X_+$ of $X$ as
\[
X_+ := \{ x \in X : \Gamma_{\epsilon_{\Gamma}}(x) \text{ is finite} \}.
\]
The \textbf{thin part} $X_-$ is its complement, i.e.
\[
X_- := \{ x \in X : \Gamma_{\epsilon_{\Gamma}}(x) \text{ is infinite} \}.
\]
We define the \textbf{thick part} $M_+$ and the \textbf{thin part} $M_-$ of $M = X\slash\Gamma$ as the quotient of the thick part and the thin part of $X$, respectively, by the group action of $\Gamma$, i.e.
\[
M_+ := X_+\slash\Gamma \qquad\text{and}\qquad M_- := X_-\slash\Gamma.
\]
By conjugation-invariance of $\gamma \mapsto \epsilon_{\gamma}$, this is well-defined.

An important fact is that in the present situation, for a fixed lattice $\Gamma$, there always exists an upper bound $\eta := \eta(\Gamma) \in \N$ on the order of finite subgroups of $\Gamma$ (see \cite{Bowditch} Proposition 5.4.2), i.e. $|G| \leq \eta$ for all finite $G < \Gamma$.

As the lengthy proof of the thick-thin decomposition is of limited benefit, we will omit it. The interested reader might look up the details in \cite{SenskaDissertation} chapter 3.2\footnote{The proof is -- as a gross simplification -- the combination of the ideas of the thick-thin decomposition for orbifolds with constant levels $\epsilon$ as in \cite{Bowditch} and the thick-thin decomposition with variable levels $\epsilon_{\gamma}$ (for manifolds) as in \cite{Sauer} and \cite{SenskaVis}.}. Eventually, we arrive at the following result:

\begin{satz}\label{Satz:Dick_duenn_Zerlegung_Orbifaltigkeit}
We have:

\begin{enumerate}
\item $M_+$ is a compact orbifold with boundary.

\item $M_+$ is connected for dimension $n>2$.

\item The number of connected components of $M_-$ is bounded by $C \cdot \Vol(M)$, where $C = C(\epsilon, n, \eta) > 0$ is a constant only depending on $\epsilon$, $n$ and $\eta = \eta(\Gamma) \in \N$. In dimension $n=2$, $C$ is independent of $\eta$.

\item The connected components $U$ of $M_-$ are of one of the following two shapes:
\begin{itemize}
\item Tubes (bounded components), i.e. (type 1) $U$ is homeomorphic to a $(D^{n-1}\slash E)$-bundle over $\S^1$ with $E < O(n-1)$ finite; or (type 2) $U$ is homeomorphic to
\[
D' \times [0,1] \quad\text{or}\quad D' \times (0,1),
\]
where $D' := (D^{n-1}\slash E)\slash \Z_2$ with finite $E < O(n-1)$, and the $\Z_2$-action on $D^{n-1}\slash E$ might be trivial.

Type 1 tubes or homotopy equivalent to $\S^1$, whereas type 2 tubes are contractible
\item Cusps (unbounded components), i.e. $U$ is homeomorphic to $V \times (0,\infty)$ for some compact $(n-1)$-dimensional orbifold $V$. If $\partial U$ is the boundary of $U$ in $M$, then there is a strong deformation retraction of $U$ onto $\partial U$.
\end{itemize}
In particular, $M$ is homotopy equivalent to the compact orbifold $M_C$ with boundary, which is constructed out of $M$ by contracting the cusps onto their common boundary with $M_+$. Equivalently, $M_C$ is the union of $M_+$ with the finitely many tubes.
\end{enumerate}
\end{satz}

It can be shown that using $M(n) := 2m(n)+1$ (with $m(n)$ the Margulis index constant), if the group
\[
\Gamma_{\epsilon/M(n)}(x) = \langle \gamma \in \Gamma : d_{\gamma}(x) < \epsilon/M(n) \rangle
\]
is infinite (where $\epsilon$ s.t. $0 < \epsilon/M(n) \leq \epsilon(n)$), there already has to be some $\gamma \in \Gamma_{\epsilon/M(n)}(x)$ of infinite order with $d_{\gamma}(x) < \epsilon$. Using this fact, we can deduce the following corollary.

\begin{kor}\label{Kor:Dick_duenn_Zerlegung_Orbifaltigkeit_gute_Niveaus}
If for some fixed $\epsilon' \in (0, \epsilon(n)/2]$ we choose
\[
\epsilon_{\gamma} := \begin{cases}
\epsilon' & \text{if } \gamma \text{ hyperbolic or parabolic}, \\
\epsilon := \frac{\epsilon'}{M(n)} & \text{if } \gamma \text{ elliptic},
\end{cases}
\]
with $M(n) \in \N$ as above, then the constant $C > 0$ in Theorem \ref{Satz:Dick_duenn_Zerlegung_Orbifaltigkeit} point 3. depends only on $\epsilon'$, $n$ and $\eta$, i.e. $C = C(\epsilon', n, \eta)$, and we have
\begin{align*}
X_- &= \{ x \in X : \text{There is a hyperbolic or parabolic } \gamma \in \Gamma \text{ with } d_{\gamma}(x) < \epsilon_{\gamma} \} \\
&= \bigcup_{\gamma \in \Gamma'} \{ d_{\gamma} < \epsilon_{\gamma} \}, \qquad \text{where } \Gamma' = \{ \gamma \in \Gamma : \gamma \text{ hyperbolic or parabolic} \}.
\end{align*}
\end{kor}

\begin{bem}
Note that for a fixed lattice $\Gamma$, there always exists a lower bound $\nu := \nu(\Gamma) > 0$ on the displacement of hyperbolic isometries, i.e. $d_{\gamma}(x) > \nu$ for all $x \in X$ and $\gamma \in \Gamma$ hyperbolic. This can be used to move the tubes to the thick part; one way of doing this is to simply assume that the Margulis $\epsilon$ is smaller than $\nu$, which effectively means replacing every occurrence of $\epsilon(n)$ by $\widetilde{\epsilon}(n)$, where $\widetilde{\epsilon}(n) := \min(\epsilon(n), \nu)$. Hence every constant depending on $\epsilon(n)$ -- in particular, the $C$ from the previous statements -- will depend on $\nu$ as well. From now on we will always assume this, i.e. every time we use the Margulis $\epsilon$, we will implicitly take its minimum with $\nu$. The benefit is that $M_-$ will then consist of cusps only (so $X_- = \bigcup_{\gamma \in \Gamma_p} \{ d_{\gamma} < \epsilon_{\gamma} \}$, where $\Gamma_p = \{ \gamma \in \Gamma : \gamma \text{ parabolic} \}$). In particular, $M_+ = M_C$ and thus $M$ itself will be homotopy equivalent to its thick part.
\end{bem}

\section{Efficient simplicial model}\label{Kapitel:Hauptresultat}

In order to achieve the desired bounds on the homology of the orbifold $M$, we will show that $M$ admits an efficient simplicial model, i.e. there is a suitable homotopy equivalence to a simplicial complex with bounded complexity. This is done in several steps. First, contracting the cusps is a straightforward way to get a homotopy equivalence between $M$ and its shrunken thick part $M'_+$\footnote{Again, using our construction in which the thin part consists solely of cusps; of course the presence of tubes in the thin part would further complicate the matter. This more general situation is dealt with in \cite{SenskaDissertation}, yielding only slightly better results than the ones stated here.}, so it suffices to construct a nice simplicial model for $M'_+$. For that, we will essentially use the nerve construction: a space with a good cover (i.e. contractible covering sets with contractible intersections) is homotopy equivalent to its nerve (simplicial) complex. The tricky part is constructing this good cover of $M'_+$. Due to the singularities in the orbifold $M$, we are no longer free to choose the position of possible covering sets. In fact, just to achieve a good cover of $M'_+$ as a subspace of $M$, the covering sets have to conform to the positions of the singular submanifolds of $M$ (to that end, we will use \textit{foldable sets} similar to those in \cite{Samet}). This rigidity in the position of the covering sets leads to a complicated boundary of the covering space (as a subspace of $M$), which will in general no longer be homotopy equivalent to $M'_+$. To remedy this, we will fill in possible gaps at the boundary of that covering space by using new sets called \textit{stretched balls}, eventually obtaining a good cover of $M'_+$ homotopy equivalent to $M'_+$.

\subsection{Defining the flow}

We often utilize that the orbifolds we study are homotopy equivalent to their thick parts. While there are many possible ways to contract onto the thick part, there is a very natural one which turns out to have many useful properties; in particular, it will define a homotopy equivalence between the thick part and the shrunken thick part as well. From now on, we will use the thick-thin decomposition (Theorem \ref{Satz:Dick_duenn_Zerlegung_Orbifaltigkeit} and Corollary \ref{Kor:Dick_duenn_Zerlegung_Orbifaltigkeit_gute_Niveaus}) with $\epsilon' := \epsilon(n)/2$. Here and henceforth, the \textbf{shrunken thick part} will be defined by
\[
X'_+ := X \setminus (X_-)_{\epsilon(n)/32} \qquad\text{and}\qquad M'_+ := M \setminus (M_-)_{\epsilon(n)/32},
\]
respectively. Recall that in our case, the thin part consists of cusps only. Now for every cusp of $M$, there is a parabolic fixed point $z \in X(\infty)$. Hence for the region around a preimage of such a cusp, we can define the flow (in $X$) by flowing along the geodesics to/from $z$. As this turns out to be $\Gamma$-equivariant, we get a similar flow in $M$. This procedure can be repeated for every single cusp separately, and since the cusps have a uniform minimal distance from each other, we can stick these flows together to construct a global flow for all cusps simultaneously. Using the notation from the thick-thin decomposition (Theorem \ref{Satz:Dick_duenn_Zerlegung_Orbifaltigkeit}), the precise statement is as follows:

\begin{lem}\label{Lem:Fluss_starke_Deformationsretraktion}
The map $F: X_+ \times [0,1] \rightarrow X_+$ -- given by flowing away from the parabolic fixed points $z \in X(\infty)$ along the geodesics to them -- defines a strong deformation retraction of $X_+$ onto $X'_+$, which is $\Gamma$-equivariant and at time $1$ induces a homeomorphism $F(\cdot,1)_{|\partial X_+}$ between $\partial X_+$ and $\partial X'_+$.

Consequently, flowing along the images (under the projection $\pi: X\rightarrow {X\slash\Gamma} = M$) of these geodesics yields a map $f: M_+ \times [0,1] \rightarrow M_+$, which is a strong deformation retraction of $M_+$ onto $M'_+$ and at time $1$ induces a homeomorphism $f(\cdot,1)_{|\partial M_+}$ between $\partial M_+$ and $\partial M'_+$.
\end{lem}

\begin{proof}
The desired properties of the flow can be shown in a similar way as in \cite{SenskaVis} chapter 4.1. A more detailed discussion can also be found in \cite{SenskaDissertation} chapter 3.
\end{proof}

\subsection{Stretched and foldable sets}

Our special tools to arrive at a good covering of the thick part in the orbifold case will be the foldable sets and stretched balls mentioned earlier. While foldable sets can be used in the more general negative curvature setting, the stretched balls are only useful in the hyperbolic case\footnote{In fact, stretched balls might also be helpful in the more general setting, but it will be much harder to show that. A crucial feature of the covering sets is that they and their intersections have to be contractible. In the hyperbolic case, it is easy to show that stretched balls are convex, immediately giving the desired properties. This uses the fact that in $\H^n$, every half space is convex; but on the other hand -- due to a classical result of Cartan -- a space where for every point $x$ and every tangent plane $\Sigma_x$ at that point there is a totally geodesic submanifold tangent to $\Sigma_x$ (which holds true under the convex half space condition) already has to have constant curvature, restricting us to the case of $\H^n$. Hence in general negative curvature, other arguments than convexity would have to be used to show that stretched balls and their intersections are contractible.}, hence from now on we will always assume $X = \H^n$.

\subsubsection*{Foldable sets}

Foldable sets were already introduced in \cite{Samet}. The motivation behind them is that in order to have a contractible image in the quotient orbifold, the shape and position of a set in the universal cover has to be compatible with the singular submanifolds\footnote{As an example, let the fixed point $p \in \H^2$ of a suitable rotation $\gamma$ be the singular submanifold. The quotient $M := \H^2\slash\langle \gamma \rangle$ will look like a cone, where $\pi(p)$ is the cone point. A ball far away from $p$ will still have an image in $M$ that looks like a ball, which thus is contractible. But as we let the ball move towards $p$, at some point the rotation will glue opposite parts of the ball together, leading to a non-contractible image homotopic to $\S^1$ -- no matter how small the radius of the ball is. Note that if the ball was centered in $p$, the image would always be contractible, as it would just be a ball around the cone point. So the moral is that (in order to have a contractible image) balls would either have to lie far away from the singular submanifolds, or -- if they lie close to them -- already be centered in them.}.

\begin{defi}
Let $U \subseteq X$ be open and $Y \subseteq X$ be a convex, complete, totally geodesic submanifold. $U$ is \textbf{$Y$-foldable}, if it has the following properties:
\begin{enumerate}
\item $U$ is convex and precisely invariant under $\Gamma$, i.e. for $\gamma \in \Gamma$ we always have $\gamma U = U$ or $\gamma U \cap U = \emptyset$.
\item $Y$ is fixed pointwise by $\Gamma_U = \{ \gamma \in \Gamma : \gamma U = U \}$.
\item $\pi_Y(U) \subseteq U$, where $\pi_Y: X \rightarrow Y$ projects to the closest point in $Y$.
\item The image $\pi(U \cap Y)$ of $U \cap Y$ in $X\slash\Gamma$ is contractible.
\end{enumerate}
If $U$ is a $Y$-foldable set for suitable $Y$, we will call $U$ \textbf{foldable}. In that case, the image $\pi(U) \subseteq X\slash\Gamma$ is a \textbf{folded set}.
\end{defi}

Note that we altered the definition slightly: we only assume that $\pi(U \cap Y)$ is contractible, but not necessarily convex. Later on, $Y$ will always be given by the fixed set of a finite subgroup of $\Gamma$, hence the assumptions on $Y$ will be fulfilled automatically.

Just as in \cite{Samet}, we get the following statements.

\begin{lem}[\cite{Samet} Proposition 4.9]\label{Lem:Gefaltete_Mengen_zusammenziehbar}
Folded sets are contractible.
\end{lem}

\begin{lem}\label{Lem:Gefaltete_Mengen_U_geschnitten_Y_injektiv}
If $U$ and $Y$ are subsets of $X$, such that $U$ is precisely invariant under $\Gamma$ and $Y$ is fixed pointwise by $\Gamma_U$, then $U \cap Y$ is mapped injectively into $X\slash\Gamma$. In particular, $U \cap Y$ is homeomorphic to $\pi(U \cap Y) \subseteq X\slash\Gamma$.
\end{lem}

We see that the above Lemma \ref{Lem:Gefaltete_Mengen_U_geschnitten_Y_injektiv} holds in particular for $Y$-foldable sets $U$. Conversely, we can also use it to prove that a set is foldable: if properties 1. and 2. hold, then for property 4. the situation in $X\slash\Gamma$ can be reduced to the one in $X$, which -- in general -- is less complicated.

The following lemma tells us, in which situations \textit{ordinary} balls (and their intersections) are foldable and when an intersection of folded balls remains folded (and thus contractible). In a next step, we will extend this to our special covering sets, the \textit{stretched} balls.

\begin{lem}[\cite{Samet} Proposition 4.10]\label{Lem:Durchschnitt_Kugeln_faltbar}
Let $Y \in \Sigma_i$ (where $i = n$ -- i.e. $Y = X$ -- is also possible).
\begin{enumerate}
\item[a)] If $y \in Y \setminus S_{<i}$ and $\mu$ is sufficiently small, such that $Y$ is fixed pointwise by $\Gamma_{4\mu}(y)$, then $B_{\mu}(y)$ is $Y$-foldable.
\item[b)] Let $U_1$ be a $Y$-foldable ball as in a) and $U_2, \ldots, U_k$ foldable balls with centers $y_2, \ldots, y_k \in Y$; note that we don't assume $y_2, \ldots, y_k \in Y \setminus S_{<i}$ and hence the $U_2, \ldots, U_k$ will in general not be $Y$-foldable. If $U := \bigcap_{j=1}^k U_j \neq \emptyset$, then:
\begin{enumerate}
\item[1.] $U$ is $Y$-foldable.
\item[2.] If the radii $\mu_j$ of the $U_j$ are chosen in a way such that $y_j$ is fixed by $\Gamma_{4\mu_j}(y_j)$ ($j=1,\ldots,k$), then $\pi(U) = \bigcap_{j=1}^k \pi(U_j)$.

Hence the intersection $\bigcap_{j=1}^k \pi(U_j)$ of the $\pi(U_j)$ is folded and thus contractible.
\end{enumerate}
\end{enumerate}
\end{lem}

\subsubsection*{Stretched balls}

We will now introduce stretched balls, which are needed later on to fill in gaps in the good cover of the thick part, which could appear if we used ordinary balls exclusively. Stretched balls only need to be defined near the common boundary of a cusp and the thick part. Let $z \in X(\infty) = \partial \H^n$ be the corresponding parabolic fixed point of the cusp and take the upper half space model of $X = \H^n$ with $z$ as $\infty$.

Denote the maximal parabolic subgroup of $\Gamma$ corresponding to $z$ by $G'$ (i.e. $G' = \Gamma_z$). If $G < \Gamma$ is a finite group fixing $z$ and $\FIX(G) = \bigcap_{g \in G} \FIX(g)$ its singular submanifold, we have $\FIX(G) \in \Sigma(G')$. Note that $\FIX(G)$ is either equal to $\H^n$ (namely if $G = \{ \id \}$) or given by the (non-empty) intersection of Euclidean affine hyperplanes perpendicular to the boundary $\partial \H^n = \R^{n-1} \times \{ 0 \}$. In the present situation we say that the singular submanifold $\FIX(G)$ contains the parabolic fixed point $z$.

Let $c: \R \rightarrow \R^n$ be the hyperbolic geodesic with endpoints $c(\infty) = z$ and $c(-\infty) = (x^{(0)}, 0)$ for some $x^{(0)} \in \R^{n-1}$, so $c([t'_0, t'_1])$ (for $t'_1 \geq t'_0 > 0$) is a geodesic section. A \textbf{stretched ball} $U$ of Euclidean radius $r > 0$ along this geodesic section is then defined by
\[
U := \{ (x, t) \in \H^n : d_{\text{eucl}}( (x, t), c(t') ) < r \text{ for some } t' \in [t'_0, t'_1] \},
\]
where $d_{\text{eucl}}$ denotes the Euclidean distance in $\R^n$. Since Euclidean balls are just hyperbolic balls (with different radius and center), we can also see such a $U$ as a union of hyperbolic balls around points of $c$. Of those points, the one with the smallest $t$-coordinate will be the (hyperbolic) initial center, whereas the one with the largest $t$-coordinate will be the (hyperbolic) end center. Obviously, the hyperbolic radius decreases monotonically while going from the initial center to the end center. The (hyperbolic) radius at the (hyperbolic) initial center will be called the (hyperbolic) initial radius, and similarly we get the (hyperbolic) end radius; the corresponding balls will be the (hyperbolic) initial ball and the (hyperbolic) end ball, respectively. Note that the hyperbolic initial and end center do not coincide with the Euclidean initial and end center $c(t'_0)$ and $c(t'_1)$, respectively; despite this, the hyperbolic initial and end balls are the same as the Euclidean initial and end balls. For a better understanding, the construction is pictured in Figure \ref{Bild_Parabolisch_gestreckte_Kugel_Konstruktion}. As a union of open balls, $U$ is itself open. $U$ is also convex, as it is the intersection of its supporting half spaces (see \cite{CEM} Proposition II.1.4.1). Finally, note that ordinary balls are special cases of stretched balls for $t'_0 = t'_1$.

{
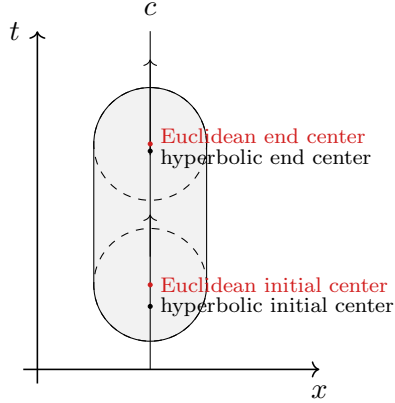
\begin{figure}
\centering

\begin{tikzpicture}[scale=0.75]




\draw[semithick, ->] (-0.25,0) -- (5,0);				
\node at (5,-0.4) {$x$};
\draw[semithick, ->] (0,-0.25) -- (0,6);				
\node at (-0.4,6) {$t$};


\fill[color=gray,opacity=0.1] (1,1.5) arc (180:360:1cm) -- (3,4) arc (0:180:1cm) -- cycle;
\draw[thin] (1,1.5) arc (180:360:1cm) -- (3,4) arc (0:180:1cm) -- cycle;


				\draw[->] (2,0) -- (2,2.75);
				\draw[->] (2,2) -- (2,5.5);
				\draw (2,3.8) -- (2,6);
				\node at (2,6.4) {$c$};

\draw[thin, dashed] (2,1.5) circle (1cm);				
\draw[thin, dashed] (2,4) circle (1cm);					

\filldraw (2,1.118) circle (1pt) node[right] {\scriptsize hyperbolic initial center};						
\filldraw (2,3.873) circle (1pt) node[right, yshift=-3pt] {\scriptsize hyperbolic end center};						

\filldraw[color=dunkelrot] (2,1.5) circle (1pt) node[right] {\scriptsize Euclidean initial center};			
\filldraw[color=dunkelrot] (2,4) circle (1pt) node[right, yshift=3pt] {\scriptsize Euclidean end center};													

\end{tikzpicture}

\caption{Construction of a stretched ball in the upper half space model of $\H^n$. The dashed balls are the initial ball (bottom) and the end ball (top).}
\label{Bild_Parabolisch_gestreckte_Kugel_Konstruktion}

\end{figure}
}

Recall that $G'$ denotes the maximal parabolic subgroup of $\Gamma$ corresponding to $z$. We will always assume that $U$ is entirely contained in the $\epsilon(n)$-thin part of $\H^n$ with respect to $G'$, i.e. $G'_{\epsilon(n)}(x) = \langle g \in G' : d_g(x) < \epsilon(n) \rangle$ is infinite for all $x \in U$. This way, we can reduce the group action of $\Gamma$ on and around $U$ to the action of $G'$, which behaves nicely with respect to the construction of $U$: since the geodesics going to $z$ -- which fiber $U$ -- are permuted by the elements of $g \in G'$ and $G$ preserves the horospheres around $z$, we see that $g U$ is also a stretched ball. Now $g U$ has the same radius as $U$; additionally, if $y$ denotes the initial center of $U$, then $g y$ is the initial center of $g U$, and $y$ and $g y$ lie in the same horosphere around $z$.

If the initial centers of two intersecting stretched balls lie in the same horosphere, we get the following estimate on their distance.

\begin{lem}\label{Lem:Parabolisch_gestreckte_Kugeln_Abstand}
Let $U$ and $U'$ be stretched balls with initial balls $B_{\mu}(y)$ and $B_{\mu'}(y')$. If $y$ and $y'$ lie in the same horosphere around $z$ and $U \cap U' \neq \emptyset$, then $d(y, y') < 2\mu + 2\mu'$.
\end{lem}

\begin{proof}
Let $r$ denote the Euclidean radius in the construction of $U$ and $y = (x_y, t_y)$; the horosphere $HS$ around $z$ containing $U$ is thus given by $\R^{n-1} \times \{ t_y \}$. The hyperbolic initial ball $B_{\mu}(y)$ of $U$ can also be seen as a Euclidean ball of radius $r$ around some point $\widetilde{y} = (x_{\widetilde{y}}, t_{\widetilde{y}})$, where $x_{\widetilde{y}} = x_y$ and $t_{\widetilde{y}} > t_y$. With the usual distance formulas, we see that the hyperbolic initial radius $\mu$ is given by
\[
\mu \stackrel{!}{=} d( (x_y, t_{\widetilde{y}} + r), (x_y, t_y) ) = \ln \left( \frac{t_{\widetilde{y}} + r}{t_y} \right),
\]
so $t_{\widetilde{y}} + r = t_y \cdot e^{\mu}$. Using $t_{\widetilde{y}} > t_y$ we conclude
\[
r = t_y \cdot e^{\mu} - t_{\widetilde{y}} < t_y \cdot e^{\mu} - t_y = t_y \cdot (e^{\mu} - 1).
\]
Recall that for a point $(x, t_y) \in HS$, we have
\[
d((x, t_y), (x_y, t_y)) = 2 \cdot \arsinh \left( \frac{\| x - x_y \|}{2 \cdot t_y} \right).
\]
If $(x, t_y)$ also lies in the Euclidean $r$-ball $E$ around $y$, then $\| x - x_y \| < r < t_y \cdot (e^{\mu} - 1)$ and thus
\[
d((x, t_y), (x_y, t_y)) < 2 \cdot \arsinh \left( \frac{t_y \cdot (e^{\mu} - 1)}{2 \cdot t_y} \right) = 2 \cdot \arsinh \left( \frac{e^{\mu} - 1}{2} \right).
\]
As $\mu > 0$, also $(e^{\mu} - 1)/2 < (e^{\mu} - e^{-\mu})/2 = \sinh(\mu)$ holds. By monotonicity of $\arsinh$, we get $\arsinh \left( (e^{\mu} - 1)/2 \right) < \arsinh(\sinh(\mu)) = \mu$, so
\[
d((x, t_y), (x_y, t_y)) < 2 \cdot \arsinh \left( \frac{e^{\mu} - 1}{2} \right) < 2 \cdot \mu.
\]
Hence every point of $HS \cap E$ has a hyperbolic distance $< 2 \mu$ to $y$.

The same arguments hold for $U'$, thus every point in the Euclidean $r'$-ball $E'$ around $y'$, which also lies in the same horosphere $HS'$ around $z$ as $y'$, has hyperbolic distance $< 2 \mu'$ to $y'$; here, $r'$ denotes the Euclidean radius in the construction of $U'$.

Now observe that since $U \cap U' \neq \emptyset$, we have $E \cap E' \neq \emptyset$. By assumption, $HS \ni y$ and $HS' \ni y'$ coincide, hence there is a point $y'' \in HS \cap E \cap E'$ (also see Figure \ref{Bild_Parabolisch_gestreckte_Kugel_Abstand}). By the above arguments we have $d(y, y'') < 2\mu$ and $d(y', y'') < 2\mu'$, thus
\[
d(y, y') < 2\mu + 2\mu'.
\]

{
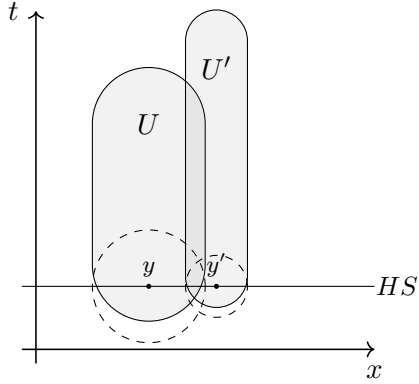
\begin{figure}
\centering

\begin{tikzpicture}[scale=0.75]




\draw[semithick, ->] (-0.25,0) -- (6,0);				
\node at (6,-0.4) {$x$};
\draw[semithick, ->] (0,-0.25) -- (0,6);				
\node at (-0.4,6) {$t$};


\fill[color=gray,opacity=0.1] (1,1.5) arc (180:360:1cm) -- (3,4) arc (0:180:1cm) -- cycle;
\draw[thin] (1,1.5) arc (180:360:1cm) -- (3,4) arc (0:180:1cm) -- cycle;


\draw[thin, dashed] (2,1.118) circle (1cm);				

\filldraw (2,1.118) circle (1pt) node[above] {\footnotesize$y$};			

\node at (2,4) {$U$};


\fill[color=gray,opacity=0.1] (2.651,1.293) arc (180:360:0.549cm) -- (3.749,5.472) arc (0:180:0.549cm) -- cycle;
\draw[thin] (2.651,1.293) arc (180:360:0.549cm) -- (3.749,5.472) arc (0:180:0.549cm) -- cycle;


\draw[thin, dashed] (3.2,1.118) circle (0.549cm);				

\filldraw (3.2,1.118) circle (1pt) node[above] {\footnotesize$y'$};			

\node at (3.2,5) {$U'$};


\draw (-0.25,1.118) -- (6,1.118);
\node at (6.4,1.118) {\small$HS$};

\end{tikzpicture}

\caption{Situation in the proof of Lemma \ref{Lem:Parabolisch_gestreckte_Kugeln_Abstand}. The dashed balls are the Euclidean $r$- and $r'$-balls $E$ and $E'$ around the hyperbolic initial centers $y$ and $y'$, respectively.}
\label{Bild_Parabolisch_gestreckte_Kugel_Abstand}

\end{figure}
}
\end{proof}

The above Lemma \ref{Lem:Parabolisch_gestreckte_Kugeln_Abstand} assumes that the initial centers lie in the \textit{same} horosphere, so we will need another construction to be able to compare stretched balls with initial centers in \textit{different} horospheres. If $B_{\mu}(y)$ is the initial ball of a stretched ball $U$ and $r$ the Euclidean radius of $U$, let $c_y$ be the geodesic from $y$ to the parabolic fixed point $z$. For a horosphere $HS'$ other than the horosphere $HS \ni y$, let $y'$ be the unique intersection of $c_y$ with $HS'$. The \textbf{comparison ball} $U'$ of $U$ at height $HS'$ is defined as the stretched ball with (hyperbolic) initial center $y'$, the same end ball as $U$ and using the same Euclidean radius $r$ in the construction, see Figure \ref{Bild_Parabolisch_gestreckte_Kugel_Vergleichskugel}.

{
\begin{figure}
\centering

\begin{tikzpicture}[scale=0.75]


\begin{scope}[xshift=-6cm]


\draw[semithick, ->] (-0.25,0) -- (4,0);				
\node at (4,-0.4) {$x$};
\draw[semithick, ->] (0,-0.25) -- (0,6);				
\node at (-0.4,6) {$t$};


\draw[->] (2,0) -- (2,0.5);
\draw[->] (2,0.5) -- (2,3.25);
\draw (2,3.25) -- (2,4);
\node at (1.6,3.6) {$c_y$};


\fill[color=gray,opacity=0.1] (1,2) arc (180:360:1cm) -- (3,5) arc (0:180:1cm) -- cycle;
\draw[thin] (1,2) arc (180:360:1cm) -- (3,5) arc (0:180:1cm) -- cycle;

\draw[thin, dashed] (2,2) circle (1cm);				

\filldraw (2,1.732) circle (1pt);			
\node at (2.18,1.92) {\footnotesize$y$};

\node at (2,5) {$U$};


\draw (-0.25,1.732) -- (4,1.732);
\node at (4.4,1.732) {\small$HS$};						

\draw (-0.25,1.118) -- (4,1.118);
\node at (4.4,1.118) {\small\,$HS'$};						
\filldraw (2,1.118) circle (1pt);
\node at (2.24,1.37) {\footnotesize$y'$};			

\draw (-0.25,2.291) -- (4,2.291);
\node at (4.4,2.291) {\small\,\,$HS''$};					
\filldraw (2,2.291) circle (1pt);
\node at (2.31,2.53) {\footnotesize$y''$};			

\end{scope}


\begin{scope}[xshift=0cm]


\draw[semithick, ->] (-0.25,0) -- (4,0);				
\node at (4,-0.4) {$x$};
\draw[semithick, ->] (0,-0.25) -- (0,6);				
\node at (-0.4,6) {$t$};


\fill[color=gray,opacity=0.1] (1,1.5) arc (180:360:1cm) -- (3,5) arc (0:180:1cm) -- cycle;
\draw[thin] (1,1.5) arc (180:360:1cm) -- (3,5) arc (0:180:1cm) -- cycle;

\draw[thin, dashed] (2,1.5) circle (1cm);				

\filldraw (2,1.118) circle (1pt) node[above] {\footnotesize$y'$};			

\node at (2,5) {$U'$};


\draw (-0.25,1.118) -- (4,1.118);
\node at (4.4,1.118) {\small\,$HS'$};

\end{scope}


\begin{scope}[xshift=6cm]


\draw[semithick, ->] (-0.25,0) -- (4,0);				
\node at (4,-0.4) {$x$};
\draw[semithick, ->] (0,-0.25) -- (0,6);				
\node at (-0.4,6) {$t$};


\fill[color=gray,opacity=0.1] (1,2.5) arc (180:360:1cm) -- (3,5) arc (0:180:1cm) -- cycle;
\draw[thin] (1,2.5) arc (180:360:1cm) -- (3,5) arc (0:180:1cm) -- cycle;

\draw[thin, dashed] (2,2.5) circle (1cm);				

\filldraw (2,2.291) circle (1pt) node[above] {\footnotesize$y''$};			

\node at (2,5) {$U''$};


\draw (-0.25,2.291) -- (4,2.291);
\node at (4.4,2.291) {\small\,\,$HS''$};

\end{scope}

\end{tikzpicture}

\caption{In the left picture, we see the stretched ball $U$ with initial center $y \in HS$ as well as the intersections $y' \in HS'$ and $y'' \in HS''$ of $c_y$; these are used as the initial centers of the comparison balls $U'$ and $U''$, respectively. The dashed balls are the initial balls of $U$, $U'$ and $U''$.
}
\label{Bild_Parabolisch_gestreckte_Kugel_Vergleichskugel}

\end{figure}
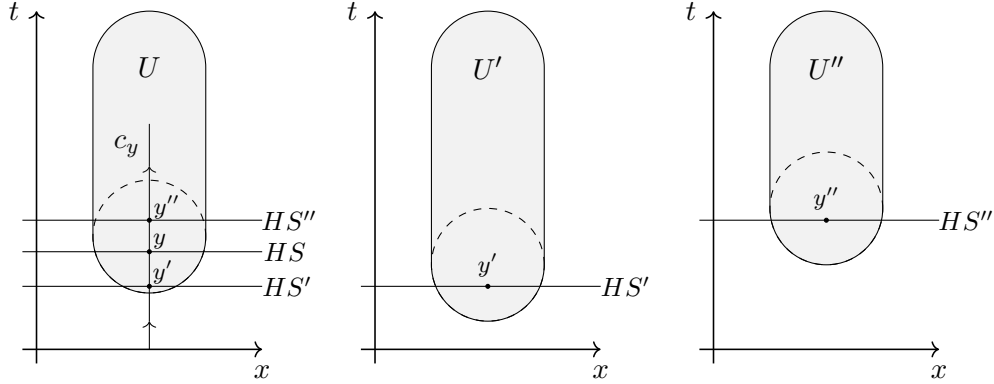
}

In other words, we just choose the parameter $t'_0$ differently. In our later applications, the comparison ball $U'$ will always be well-defined, i.e. the new initial center $y'$ will always be farther away from $z$ than the end center (which coincides with the end center of $U$).

We immediately see that the comparison ball $U'$ intersects precisely those geodesics to $z$ that also intersect with $U$. Since the distance between these geodesics decreases on its way to $z$, we get the following estimates for the hyperbolic initial radius $\mu'$ of $U'$: if $HS'$ is closer to $z$ than $HS$, we have $\mu' \leq \mu$; whereas if $HS'$ is farther away from $z$ than $HS$, we have $\mu' \geq \mu$. Also note that $U$ intersects another stretched ball $U''$ with initial center $y'' \in HS''$ if and only if the comparison ball $U'$ of $U$ at height $HS''$ intersects $U''$.

Similar to the case of ordinary balls (compare Lemma \ref{Lem:Durchschnitt_Kugeln_faltbar}), suitably chosen stretched balls are foldable:

\begin{lem}\label{Lem:Parabolisch_gestreckte_Kugeln_faltbar}
Let $Y \in \Sigma_i$ be a singular submanifold containing the parabolic fixed point $z$. Moreover, let $y \in Y \setminus S_{<i}$ and $\mu > 0$ be sufficiently small, such that $Y$ is fixed pointwise by $\Gamma_{4\mu}(y)$. If $U$ denotes a stretched ball with (hyperbolic) initial ball $B_{\mu}(y)$, then $U$ is $Y$-foldable.
\end{lem}

\begin{proof}
We have already seen that stretched balls are open and convex, so let's turn to the precise invariance of $U$. Since we always assume that $U$ lies in the $\epsilon(n)$-thin part of $\H^n$ w.r.t. $G$ (where $G$ is the maximal parabolic subgroup corresponding to the parabolic fixed point $z$) -- and that part is (seen as a component of the $\epsilon(n)$-thin part w.r.t. $\Gamma$) precisely invariant under $\Gamma$ --, it only remains to show the precise invariance of $U$ w.r.t. elements of $G$. Hence, we have to check if for all $g \in G$, the condition $g U \cap U \neq \emptyset$ already implies $g U = U$. As $g \in G$, we know that $gU$ is a stretched ball with same initial radius as $U$ and initial center $gy$ in the same horosphere around $z$ as $y$. Since $gU \cap U \neq \emptyset$, we can apply Lemma \ref{Lem:Parabolisch_gestreckte_Kugeln_Abstand} and get
\[
d_g(y) = d(y, gy) < 2\mu + 2\mu = 4\mu,
\]
i.e. $g \in \Gamma_{4\mu}(y)$. By assumption, this means that $g$ fixes $Y$ pointwise, thus $gy = y$ and hence $gU = U$.

As a next step, we will prove that $Y$ is fixed pointwise by $\Gamma_U$. Note that in our situation, we have $\Gamma_U = G_U$, where $G_U = \{ g \in G : gU = U \}$. By the above arguments we know that $g \in G_U$ already implies $d_g(y) < 4\mu$, so $g \in \Gamma_{4\mu}(y)$; hence by the assumption, $g$ fixes $Y$ pointwise.

We will now prove $\pi_Y(U) \subseteq U$. As every $u \in U$ lies in a suitable hyperbolic ball $B_{\mu_0}(y_0) \subseteq U$ of radius $\mu_0$ around some $y_0 \in Y$ (recall that the geodesic $c$ in the construction of $U$ is entirely contained in $Y$), we get (using that the projection to $Y$ is distance-decreasing, compare \cite{BGS} chapter 1.6)
\[
d(\pi_Y(u), \pi_Y(y_0)) \leq d(u, y_0) < \mu_0.
\]
Since $\pi_Y(y_0) = y_0$, this implies $\pi_Y(u) \in B_{\mu_0}(y_0) \subseteq U$.

In the last step, we have to show that $\pi(U \cap Y) \subseteq \H^n\slash\Gamma$ is contractible; by Lemma \ref{Lem:Gefaltete_Mengen_U_geschnitten_Y_injektiv}, this is equivalent to $U \cap Y \subseteq \H^n$ being contractible. Now recall that $U$ and $Y$ are convex, so $U \cap Y$ is convex and thus contractible.
\end{proof}

We also need a similar statement for the intersection of (the images of) several stretched balls:

\begin{lem}\label{Lem:Durchschnitt_parabolisch_gestreckte_Kugeln_faltbar}
Let $U_1$ be a $Y$-foldable stretched ball as in Lemma \ref{Lem:Parabolisch_gestreckte_Kugeln_faltbar}, i.e. with initial center $y_1 \in Y \setminus S_{<i}$, and let $U_2, \ldots, U_k$ be foldable stretched balls with initial centers $y_2, \ldots, y_k \in Y$; note that we don't assume $y_2, \ldots, y_k \in Y \setminus S_{<i}$ and hence the $U_2, \ldots, U_k$ will in general not be $Y$-foldable. If $U := \bigcap_{j=1}^k U_j \neq \emptyset$, then:
\begin{enumerate}
\item $U$ is $Y$-foldable.
\item If the hyperbolic initial radii $\mu_j$ of the $U_j$ at the initial centers $y_j$ are chosen in a way such that $y_j$ is fixed by $\Gamma_{8\mu_j}(y_j)$ ($j=1,\ldots,k$), then $\pi(U) = \bigcap_{j=1}^k \pi(U_j)$.

Hence the intersection $\bigcap_{j=1}^k \pi(U_j)$ of the $\pi(U_j)$ is folded and thus contractible.
\end{enumerate}
\end{lem}

\begin{proof}
Many ideas are similar to those in the proof of \cite{Samet} Proposition 4.10 and its preceding text.
\begin{enumerate}
\item As an intersection of open, convex and precisely invariant sets, $U$ itself is open, convex and precisely invariant.

Since $\Gamma_{U_1}$ fixes $Y$ pointwise (because $U_1$ is $Y$-foldable), we only need to show that $\Gamma_U$ is a subgroup of $\Gamma_{U_1}$ to conclude that $\Gamma_U$ fixes $Y$ pointwise. Let $\gamma \in \Gamma_U$. As $\gamma U = U$ and $U \subseteq U_1$, we see that
\[
\gamma U_1 \supseteq \gamma U = U \subseteq U_1,
\]
so $\gamma U_1 \cap U_1 \neq \emptyset$ (because $U \neq \emptyset$). Knowing that $U_1$ is precisely invariant, this leads to $\gamma U_1 = U_1$, hence $\gamma \in \Gamma_{U_1}$.

Next, let us check if $\pi_Y(U) \subseteq U$. Let $c_j$ denote the geodesic from the initial center $y_j \in Y$ of $U_j$ to the parabolic fixed point $z$; note that all the $c_j$ are entirely contained in $Y$. Hence the $U_j$ are constructed as the union of hyperbolic balls around points of $c_j \subseteq Y$. So if $u \in U$, there are points $y_0^{(j)} \in Y$ and radii $\mu_0^{(j)}$ ($j=1,\ldots,k$) such that $u$ lies in every ball $B_{\mu_0^{(j)}}(y_0^{(j)}) =: B_j \subseteq U_j$. Just as in the proof of Lemma \ref{Lem:Parabolisch_gestreckte_Kugeln_faltbar} we conclude that $\pi_Y(u) \in B_j$ for all $j=1,\ldots,k$. Thus
\[
\pi_Y(u) \in \bigcap_{j=1}^k B_j \subseteq \bigcap_{j=1}^k U_j = U,
\]
proving $\pi_Y(U) \subseteq U$.

It remains to show that $\pi(U \cap Y)$ is contractible. Again, this is equivalent to $U \cap Y$ being contractible (by Lemma \ref{Lem:Gefaltete_Mengen_U_geschnitten_Y_injektiv}), which itself is a consequence of the convexity of $U$ and $Y$.

\item Similar to \cite{Samet} Proposition 4.10, we see that a preimage of $\bigcap_{j=1}^k \pi(U_j)$ in $\H^n$ is just a union of intersections $\bigcap_{j=1}^k \gamma_j U_j$ ($\gamma_j \in \Gamma$), so it only remains to show that every such non-empty intersection $\bigcap_{j=1}^k \gamma_j U_j$ arises as a translate of $\bigcap_{j=1}^k U_j$ under a suitable element of $\Gamma$.

Let $r_j$ be the Euclidean radius in the construction of $U_j$ and choose $j_0 \in \{ 1, \ldots, k \}$ such that $r_{j_0} \leq r_j$ for all $j=1,\ldots,k$. After a possible translation of $\bigcap_{j=1}^k \gamma_j U_j$ by $\gamma_{j_0}^{-1}$, we can assume that $\gamma_{j_0} = \id$. As $\bigcap_{j=1}^k \gamma_j U_j \neq \emptyset$ by assumption, $U_{j_0}$ intersects all the other $\gamma_j U_j$.

We let $U_{j_0}^{(j)}$ denote the comparison ball of $U_{j_0}$ at the height of the horosphere $HS_j \ni y_j$ (for $j=1,\ldots,k$), and $\mu_{j_0}^{(j)}$ the (hyperbolic) initial radius around the (hyperbolic) initial center $y_{j_0}^{(j)}$. Since $U_{j_0}$ intersects $U_j$ and $\gamma_j U_j$, also $U_{j_0}^{(j)}$ intersects $U_j$ and $\gamma_j U_j$ (see definition of the comparison balls). As the Euclidean radius $r_{j_0}$ was chosen to be minimal -- and the Euclidean radius of the comparison ball coincides with $r_{j_0}$ --, the hyperbolic initial radius $\mu_{j_0}^{(j)}$ also has to satisfy $\mu_{j_0}^{(j)} \leq \mu_j$ for all $j=1,\ldots,k$. Using Lemma \ref{Lem:Parabolisch_gestreckte_Kugeln_Abstand}, we get $d(y_j, y_{j_0}^{(j)}) < 2\mu_j + 2\mu_{j_0}^{(j)} \leq 4\mu_j$ and $d(\gamma_j y_j, y_{j_0}^{(j)}) < 2\mu_j + 2\mu_{j_0}^{(j)} \leq 4\mu_j$, thus
\[
d_{\gamma_j}(y_j) = d(y_j, \gamma_j y_j) < 8\mu_j
\]
for every $j$. Hence $\gamma_j \in \Gamma_{8\mu_j}(y_j)$ and so by assumption, $\gamma_j$ fixes $y_j$. We conclude $\gamma_j U_j = U_j$ and thus
\[
\bigcap_{j=1}^k \gamma_j U_j = \bigcap_{j=1}^k U_j.
\]
So $\bigcap_{j=1}^k \pi(U_j) = \pi(U)$, the latter set being the image of the (by 1.) $Y$-foldable set $U$. Hence $\bigcap_{j=1}^k \pi(U_j)$ is folded and thus contractible.
\end{enumerate}
\end{proof}

\subsection{Constructing the cover}

Our next goal is to define a suitable cover of the thick part $M_+$. To this end, we will extend the construction of Samet \cite{Samet} Theorem 4.2: while the cover given there is indeed a good cover and contains $M_+$, it is far from being homotopy equivalent to $M_+$; it goes well beyond $M_+$ and, in general, will have gaps outside $M_+$. By gaps we mean that coming from the thin part (in a suitable way), we might enter and leave the cover several times before entering it for a last time and staying in $M_+$. To fill these gaps, we will use the previously defined stretched balls, eventually giving us a good cover that is also homotopy equivalent to $M_+$.

To achieve all the said properties, a very delicate choice of positions and sizes of the covering sets will be required; the following two lemmas are a major tool for this.

\begin{lem}[\cite{Samet} Proposition 4.6]\label{Samet4.6}
For every $\epsilon_1 > 0$ there is $\epsilon_2 = \epsilon_2(\epsilon_1) > 0$ with the following property. Let $Y_1, Y_2 \in \Sigma(\Gamma)$ with $i = \dim(Y_2) \leq \dim(Y_1)$ and $y_j \in Y_j \cap X_+$ ($j=1,2$), such that $d(y_1, S_{<i}(\Gamma)) \geq \epsilon_1$ or $d(y_2, S_{<i}(\Gamma)) \geq \epsilon_1$. If $d(y_1, y_2) < \epsilon_2$, then already $Y_2 \subseteq Y_1$.
\end{lem}

\begin{lem}[\cite{Samet} Proposition 4.7]\label{Samet4.7}
For every $\epsilon_1 > 0$ there is $\epsilon_3 = \epsilon_3(\epsilon_1) > 0$ with the following property. Let $Y \in \Sigma(\Gamma)$, $y \in Y \cap X_+$ and $i = \dim(Y)$. If $d(y, S_{<i}(\Gamma)) > \epsilon_1$, then $Y$ is fixed pointwise by every element of $\Gamma_{\epsilon_1}(y)$.
\end{lem}

We have let our $X_+$ take the role of $X_{\geq\epsilon, m}$ in \cite{Samet}, so we will always have to assume (without restriction) that $\epsilon_2(\cdot), \epsilon_3(\cdot) \leq \epsilon(n)/(2 M(n))$; recall that $X_+ \subseteq X_{\geq\epsilon, \eta}$ with $\epsilon = \epsilon(n)/(2M(n))$. Define
\[
\mu_{-1} := \min\left( \frac{\epsilon(n)}{64}, \nu \right),
\]
where $\nu$ is the minimal displacement of hyperbolic isometries of $\Gamma$. We then iteratively define $\mu_{-1} > \mu_0 > \ldots > \mu_n$ by
\[
\mu_{i+1} := \min \left( \frac{\epsilon_2(\mu_i)}{12}, \frac{\epsilon_3(\mu_i)}{24}, \frac{\mu_i}{12} \right),
\]
with $\epsilon_2(\cdot)$ and $\epsilon_3(\cdot)$ given by Lemma \ref{Samet4.6} and \ref{Samet4.7}. Furthermore, let
\begin{align*}
\mathcal{D}_0 &:= \text{ maximal } \mu_0\text{-discrete subset in } \overline{(M'_+)_{8\mu_0}} \cap \pi(S_0), \quad\text{and} \\
\mathcal{D}_i &:= \text{ maximal } \mu_i\text{-discrete subset in } \big( \overline{(M'_+)_{8\mu_i}} \cap \pi(S_i) \big) \setminus \bigcup_{j<i} (\pi(S_j))_{\mu_j}
\end{align*}
for $i > 0$. Without restriction, we can assume that $\mathcal{D}_0$ contains a maximal $\mu_0$-discrete subset of
\[
\partial (M'_+)_{8\mu_0} \cap \pi(S_0) \subseteq \overline{(M'_+)_{8\mu_0}} \cap \pi(S_0),
\]
and similarly $\mathcal{D}_i$ (for $i>0$) contains a maximal $\mu_i$-discrete subset of
\[
\big( \partial (M'_+)_{8\mu_i} \cap \pi(S_i) \big) \setminus \bigcup_{j<i} (\pi(S_j))_{\mu_j} \subseteq \big( \overline{(M'_+)_{8\mu_i}} \cap \pi(S_i) \big) \setminus \bigcup_{j<i} (\pi(S_j))_{\mu_j},
\]
because maximal $\mu_i$-discrete subsets in the sets on the left hand side can be extended to maximal $\mu_i$-discrete subsets in the sets on the right hand side. Denote by
\[
\mathcal{D} := \bigcup_{i=0}^n \mathcal{D}_i
\]
the set of all centers. In a first step, let
\[
\mathcal{B}' := \{ B_{3\mu_i}^M(x) : x \in \mathcal{D}_i \text{ for some } i \in \{ 0, \ldots, n \} \}
\]
be the set of all ordinary balls (in $M$) around the points of $\mathcal{D}$. The following lemma says that this already covers the shrunken thick part.

\begin{lem}\label{Lem:Ueberdeckung_ueberdeckt}
The sets of $\mathcal{B}'$ form a cover of $M'_+$, i.e.
\[
M'_+ \subseteq \bigcup_{B \in \mathcal{B}'} B.
\]
\end{lem}

\begin{proof}
The proof contains ideas of \cite{Samet} Theorem 4.2 step 1). Let $x \in M'_+$ and choose $i \in \{0, \ldots, n \}$ minimal, such that $d(x, \pi(S_i)) < 2\mu_i$; since $S_n = X$ -- i.e. $\pi(S_n) = M$ --, this is always possible. Hence there is $y \in \pi(S_i)$ with $d(x,y) < 2\mu_i$.

We claim that $y \notin \bigcup_{j<i} (\pi(S_j))_{\mu_j}$. Assume the contrary, then there would be $j < i$ and $z \in \pi(S_j)$ with $d(y,z) < \mu_j$. As $\mu_i \leq \mu_j/12$ (because $j<i$), this leads to
\[
d(x,z) \leq d(x,y) + d(y,z) < 2\mu_i + \mu_j < 2\mu_j,
\]
i.e. $d(x, \pi(S_j)) < 2\mu_j$; since $j < i$, this contradicts the minimality of $i$. Thus $y \notin \bigcup_{j<i} (\pi(S_j))_{\mu_j}$.

Using $d(x,y) < 2\mu_i < 8\mu_i$, we also get $d(y, M'_+) \leq d(y, x) < 8\mu_i$, so $y \in \overline{(M'_+)_{8\mu_i}}$ and hence
\[
y \in \big( \overline{(M'_+)_{8\mu_i}} \cap \pi(S_i) \big) \setminus \bigcup_{j<i} (\pi(S_j))_{\mu_j}.
\]
By definition, $\mathcal{D}_i$ lies in that set as a maximal $\mu_i$-discrete subset, so there must be a $y' \in \mathcal{D}_i$ with $d(y, y') < \mu_i$. We conclude
\[
d(x,y') \leq d(x,y) + d(y,y') < 2\mu_i + \mu_i = 3\mu_i,
\]
i.e. $x \in B_{3\mu_i}^M(y') \in \mathcal{B}'$, which finishes the proof.
\end{proof}

To fill in possible gaps of the cover outside of $M_+$ (or using another interpretation: to make the cover stable\footnote{We say that a set is stable under the flow if the flow does not leave that set after entering it for the first time.} under the flow to $M_+$), we will have to stretch some of the balls of the cover; this simply means that we replace that ordinary ball by (the image of) a suitable stretched ball with that ball as its initial ball. Recall that in the present situation, $\pi(B_r^X(y)) = B_r^M(\pi(y))$ for all $y \in X$ and $r > 0$, so for the sake of simplicity, the image of a stretched ball in $X$ will also be called stretched ball (in $M$).

If $x \in \mathcal{D}_i \cap \partial (M'_+)_{8\mu_i}$, then (by choice of $M'_+$) $B := B_{3\mu_i}^M(x)$ is contained in the $\epsilon(n)$-thin part of $M$. We will replace such a ball by a corresponding stretched ball; at this point, the length of this stretching is of no further importance and will be specified later on in the proof of Lemma \ref{Lem:Laenge_Streckung} (see below). These stretchings will turn the set $\mathcal{B}'$ into our final set of covering sets $\mathcal{B}$. Observe that the set $\mathcal{D}_i \cap \partial (M'_+)_{8\mu_i}$ of initial centers of the stretched balls contains precisely the points of $\mathcal{D}_i$ with maximal distance to $M'_+$ (namely $8\mu_i$).

With this choice of stretchings, we have achieved a monotonicity of the radii of the (stretched) balls; in simple terms, it means that the initial centers of stretched balls with large radius are closer to the thin part than the (initial) centers of arbitrary balls of smaller radius (if these balls intersect):

\begin{lem}\label{Lem:Monotonie_der_Radien}
Let $y_i \in \widetilde{\mathcal{D}_i}$ be an initial center of a stretched ball $B_i \in \widetilde{\mathcal{B}}$ with initial ball $B_{3\mu_i}^X(y_i)$ and $B_j \in \widetilde{\mathcal{B}}$ be an arbitrary (i.e. ordinary or stretched) ball with (initial) center $B_{3\mu_j}^X(y_j)$ for $y_j \in \widetilde{\mathcal{D}_j}$. By $HS_i \ni y_i$ and $HS_j \ni y_j$ we denote the corresponding horospheres around $z$. If $B_i \cap B_j \neq \emptyset$ and $3\mu_i > 3\mu_j$ (so equivalently: $i<j$), then $HS_i$ is closer to $z$ than $HS_j$.
\end{lem}

\begin{proof}
We already see from the formulation of the statement that we will treat the situation in $X$, where $\widetilde{\mathcal{D}_i}$ and $\widetilde{\mathcal{B}}$ denote the lifts (to $X$) of the corresponding sets in $M$.

By the choice of centers, we have $d(y_i, X'_+) = 8\mu_i$; moreover, $d(y_j, X'_+) \leq 8\mu_j$. Let $x \in B_i \cap B_j$. If the (initial) balls $B_{3\mu_i}^X(y_i)$ and $B_{3\mu_j}^X(y_j)$ would intersect (in $x$, without restriction), then we would reach the contradiction
\[
8 \mu_i = d(y_i, X'_+) \leq d(y_i, x) + d(x, y_j) + d(y_j, X'_+) < 3\mu_i + 3\mu_j + 8\mu_j < 4\mu_i
\]
(note that $12\mu_j \leq \mu_i$, since $i<j$). Let $c_x$ be the flow geodesic of $x$, with parametrization $c_x(0) = x$ and $c_x(-\infty) = z$ (i.e. flowing to $X'_+$). We've already seen that the (initial) balls can not intersect, but $c_x$ has to intersect both (initial) balls. We will distinguish between the two possible cases:

In the first case, $c_x$ leaves the smaller (initial) ball $B_{3\mu_j}^X(y_j)$ before entering the larger initial ball $B_{3\mu_i}^X(y_i)$. Observe that all points of $B_{3\mu_j}^X(y_j)$ have distance $< 3\mu_j + 8\mu_j < 12\mu_j \leq \mu_i$ to $X'_+$, because $d(y_j, X'_+) \leq 8\mu_j$. As the distance to $X'_+$ decreases along the flow, this means that the entry point of $c_x$ in $\overline{B_{3\mu_i}^X(y_i)}$ also has distance $<\mu_i$ to $X'_+$. This is a contradicition, since $d(y_i, X'_+) = 8\mu_i$ already implies that all points $y \in \overline{B_{3\mu_i}^X(y_i)}$ satisfy $d(y, X'_+) \geq 8\mu_i - 3\mu_i = 5\mu_i$. Consequently, this case can not happen.

In the second case, $c_x$ leaves the larger initial ball $B_{3\mu_i}^X(y_i)$ before entering the smaller (initial) ball $B_{3\mu_j}^X(y_j)$. Using the above estimates for the distance to $X'_+$, we know that along $c_x$, we have to flow at least $5\mu_i - \mu_i = 4\mu_i$ after leaving $B_{3\mu_i}^X(y_i)$ before we enter $B_{3\mu_j}^X(y_j)$. If $x_i = c_x(t_i)$ denotes the exit point of $\overline{B_{3\mu_i}^X(y_i)}$ and $x_j = c_x(t_j)$ the entry point in $\overline{B_{3\mu_j}^X(y_j)}$, we see that $t_j \geq t_i + 4\mu_i$

We will now use the upper half space model (again with $z$ as point $\infty$), so horospheres around $z$ are (Euclidean) hyperplanes $\R^{n-1} \times \{t\} \subseteq \H^n$ and the flow geodesics are (Euclidean) lines going away from $\infty$ and perpendicular to $\R^{n-1} \times \{0\}$. So the exit point $x_i$ is on the lower half of the initial ball $B_{3\mu_i}^X(y_i)$ (seen as a Euclidean ball), i.e. it has a $t$-coordinate smaller than the $t$-coordinate of the Euclidean center of $B_{3\mu_i}^X(y_i)$. Assuming $y_i = (0,\ldots,0,1)$ (without restriction), by similar arguments as in the proof of Lemma \ref{Lem:Parabolisch_gestreckte_Kugeln_Abstand} and using the usual distance formulas we can deduce that the $t$-coordinate of the Euclidean center of $B_{3\mu_i}^X(y_i)$ is $\cosh(3\mu_i) > 1$. Similarly, the $t$-coordinate of a point that is reached after flowing $4\mu_i$ (in hyperbolic length) -- and starting from a $t$-coordinate of $\cosh(3\mu_i)$ -- can be computed to be $e^{-\mu_i} < 1$ (again using the usual distance formulas). By the above arguments, this value is an upper bound on the $t$-coordinate of any entry point of the ball $B_{3\mu_j}^X(y_j)$; so an arbitrary entry point of $c_x$ into $B_{3\mu_j}^X(y_j)$ has a $t$-coordinate $<1$. Since the $t$-coordinate of the center $y_j$ has an even smaller $t$-coordinate than the entry point, $y_j$ has a $t$-coordinate $<1$. Thus $y_j$ lies in a horosphere $\R^{n-1} \times \{t\}$ for some $t < 1$, whereas $y_i$ lies in the horosphere $\R^{n-1} \times \{1\}$.
\end{proof}

\subsubsection*{Stretching length}

Before proving the desired properties of the cover, we will further investigate the length of the stretched balls.

\begin{lem}\label{Lem:Laenge_Streckung}
The stretching length\footnote{As measured by the length of the (Euclidean) line segment in the construction of the stretched balls or -- equivalently -- by the (hyperbolic) distance between the initial center and the end center.} of the balls can be chosen in such a way that all the sets of $\mathcal{B}$ are contained in the (ordinary) thick part, i.e.
\[
\bigcup_{B \in \mathcal{B}} B \subseteq M_+.
\]
\end{lem}

\begin{proof}
The general situation and strategy is summarized in Figure \ref{Bild_Streckungslaenge}.

{
\begin{figure}
\centering

\begin{tikzpicture}[scale=0.75]





\fill[color=gray, opacity=0.05] (0,0) .. controls (1,0.5) and (4,0.5) .. (5,0) -- (5,-0.5) -- (0,-0.5) -- cycle;
\draw[very thin] (0,0) .. controls (1,0.5) and (4,0.5) .. (5,0) node[right] {\footnotesize$\partial M_+$};
\node at (2.5,-0.15) {\footnotesize$M_-$};


\draw[very thin, dashed] (0,2) .. controls (1,2.5) and (4,2.5) .. (5,2);


\fill[color=gray, opacity=0.05] (0,4) .. controls (1,4.5) and (4,4.5) .. (5,4) -- (5,5) -- (0,5) -- cycle;
\draw[very thin] (0,4) .. controls (1,4.5) and (4,4.5) .. (5,4) node[right] {\footnotesize$\partial M'_+$};
\node at (2.5,4.65) {\footnotesize$M'_+$};


\draw [decorate,decoration={brace,amplitude=5pt},xshift=0pt,yshift=0pt]
(0,0) -- (0,2) node [black,midway,xshift=-1cm] 
{\scriptsize $\mu_{-1} = \frac{\epsilon(n)}{64}$};

\draw [decorate,decoration={brace,amplitude=5pt},xshift=0pt,yshift=0pt]
(0,2) -- (0,4) node [black,midway,xshift=-1cm] 
{\scriptsize $\mu_{-1} = \frac{\epsilon(n)}{64}$};


\filldraw (2.5,1.375) circle (0.25pt);												
\fill[color=gray, opacity=0.1] (2.5,1.375) circle (0.5cm);		
\draw[ultra thin] (2.5,1.375) circle(0.5cm);									

\draw[decorate,decoration={brace,amplitude=3pt},xshift=0pt,yshift=0pt]
(2.5,0.375) -- (2.5,0.875) node [black,midway,xshift=-0.4cm] 
{\scriptsize $\frac{\mu_{-1}}{4}$};

\draw[decorate,decoration={brace,amplitude=3pt},xshift=0pt,yshift=0pt]
(2.5,0.875) -- (2.5,1.875) node [black,midway,xshift=-0.6cm] 
{\scriptsize $2\!\cdot\!\frac{\mu_{-1}}{4}$};

\draw[decorate,decoration={brace,amplitude=3pt},xshift=0pt,yshift=0pt]
(2.5,1.875) -- (2.5,2.375) node [black,midway,xshift=-0.4cm] 
{\scriptsize $\frac{\mu_{-1}}{4}$};


\end{tikzpicture}

\caption{Choosing a suitable stretching length in the proof of Lemma \ref{Lem:Laenge_Streckung}. It is enough to specify the position of the end center (here: the center of the gray ball). Ordinary balls of $\mathcal{B}'$ (equivalently: their preimages in $X$) can not exceed the dashed line.}
\label{Bild_Streckungslaenge}

\end{figure}
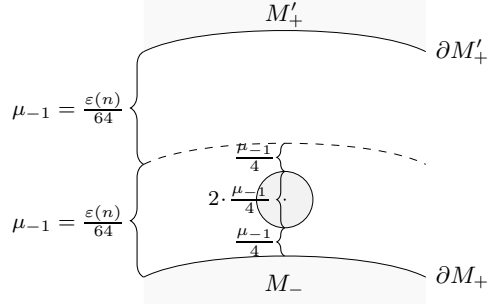
}

We first remark that by the monotonicity $\mu_i > \mu_{i+1}$, the balls of $\mathcal{B}'$ with centers $x \in \mathcal{D}_0$ are those which can lie the furthest away from $M'_+$; for these $x$ we have $d(x, M'_+) \leq 8\mu_0$, so for arbitrary $y \in B_{3\mu_0}^M(x)$ we get
\[
d(y, M'_+) < 8\mu_0 + 3\mu_0 = 11\mu_0 < 12\mu_0 \leq \mu_{-1},
\]
where $\mu_{-1} = \epsilon(n)/64$ is half the size $\epsilon(n)/32$ of the shrinking of $M_+$ onto $M'_+$ (so every point of $\bigcup_{B \in \mathcal{B}'} B$ has a distance $>\mu_{-1}$ to the boundary $\partial M_+$). This yields $\bigcup_{B \in \mathcal{B}'} B \subseteq M_+$, so it remains to show that this property is preserved when stretching the balls, i.e. when going from $\mathcal{B}'$ to $\mathcal{B}$.

Let $y$ be the center of a ball $B_{3\mu_i}^M(y) \in \mathcal{B}'$ which should be stretched; to specify the stretching length, it is enough to define the position of the end center\footnote{We stated the construction of stretched balls only in $X$, but the construction in $M$ can be translated to $X$ by lifting the end center.}. Let $c_y$ be the flow geodesic of $y$; we will flow along $c_y$ in direction of $\partial M_+$ until $c_y(t)$ is precisely $\mu_{-1}/2$ away from $\partial M_+$, but still lies in $M_+$, i.e. for this $t$ we have $c_y(t) \in M_+$ and $d(c_y(t), \partial M_+) = \mu_{-1}/2$. We claim that this $c_y(t)$ satisfies the conditions needed to be an end center. Observe that the hyperbolic radius at the end center is smaller than the hyperbolic radius at the initial center, and since
\[
3\mu_i \leq 3\mu_0 \leq \frac{3\mu_{-1}}{12} = \frac{\mu_{-1}}{4},
\]
the latter is bounded by $\mu_{-1}/4$. Hence the hyperbolic end radius is also smaller than $\mu_{-1}/4$. Thus by choice of $c_y(t)$, every point of the end ball has distance $\geq \mu_{-1}/2 - \mu_{-1}/4 = \mu_{-1}/4$ from $\partial M_+$, so the end ball itself is completely contained in $M_+$. Consequently, the entire stretched ball lies inside $M_+$. Also note that the end ball has distance $\geq \mu_{-1}/4$ to any ordinary ball of $\mathcal{B}$, because the distance of points of the ordinary balls to $\partial M_+$ is always $\geq \mu_{-1}$.
\end{proof}

The stretching length can be taken from the proof of the above Lemma \ref{Lem:Laenge_Streckung}, but there it is defined via the distance to $\partial M_+$; so in order to get a value for the stretching length in absolute terms, we have to check how long we have to flow to realize said distance.

\begin{lem}\label{Lem:Laenge_Streckung_absolut}
Every stretched ball $B \in \mathcal{B}$ with initial ball $B_{3\mu_i}^M(y)$ is contained in a ball $B_{R(n,\nu)}^M(y)$, where $R(n,\nu) > 0$ is a constant only depending on the dimension $n$ and the constant $\nu$.
\end{lem}

\begin{proof}
We will omit any further mentioning of $\nu$ in the proof, as the dependency of $R$ on $\nu$ only comes from the definition of $\mu_{-1}$ as the minimum of $\epsilon(n)/64$ and $\nu$; here, the dependency on $\epsilon(n)/64$ contains the interesting information on the correct size of $R$, whereas $\nu$ is only needed to account for the fact that we defined our thick part using $\nu$ in order to have no tubes in the thin part.

Let $c_y$ be the flow geodesic of $y = c_y(0)$. In the proof of Lemma \ref{Lem:Laenge_Streckung} we have seen that the end center $c_y(t)$ of $B$ is characterized by $c_y(t) \in M_+$ and $d(c_y(t), \partial M_+) = \mu_{-1}/2$, for suitable $t>0$. By the construction of the stretched balls, every point $y' \in B$ is contained in a ball around some $c_y(t')$, where $0 \leq t' \leq t$; the radius of that ball is at most $3\mu_i$ (which is the initial radius). The latter value is bounded by $\mu_{-1}$, so using $d(y, c_y(t')) = t'$ we get
\[
B \subseteq B_{t + \mu_{-1}}^M(y).
\]
As $\mu_{-1}$ only depends on $n$, it remains to show that $t$ can also be controlled by only $n$. Recall that $y$ is not contained in $M'_+$ (because it is the initial center of a stretched ball), but in the piece between $\partial M'_+$ and $\partial M_+$, and similarly for the end center. If we could bound the length $L$ of $c_y$ between $\partial M'_+$ and $\partial M_+$, we would thus get the desired bound on the above $t$. Note that equivalently, we can try to control the length of the lift of $c_y$ (which we also denote by $c_y$) in $X$ between $\partial X'_+$ and $\partial X_+$.

Let $y_1 \in \partial X'_+$ and $y_2 \in \partial X_+$ be the intersections of $c_y$ with $\partial X'_+$ and $\partial X_+$. We reparametrize $c_y$ such that $c_y(0) = y_1$, so we only have to find an upper bound for $t_2 > 0$, where $c_y(t_2) = y_2$. By definition of $X'_+$, we have that $y_1 \in \overline{ ( \{ d_{\gamma} < \epsilon(n)/2 \} )_{\epsilon(n)/32} }$ for some parabolic isometry $\gamma \in \Gamma$ with fixed point $z = c_y(\infty)$. Let $y_3 = c_y(t_3)$ be the intersection of $c_y$ with $\partial \{ d_{\gamma} < \epsilon(n)/2 \}$, so $t_3 \geq t_2$ holds\footnote{Either $y_3$ is the intersection of $c_y$ with $\partial X_+$ (i.e. $t_3 = t_2$), or $c_y$ entered $\partial X_+$ earlier (at the boundary of another sublevel set, i.e. $t_2 \leq t_3$).}.

Thus it remains to show: if $y_1 = c_y(0)$ is the entry point of $c_y$ in $\overline{ ( \{ d_{\gamma} < \epsilon(n)/2 \} )_{\epsilon(n)/32} }$ and $y_3 = c_y(t_3)$ is the entry point of $c_y$ in $\overline{ \{ d_{\gamma} < \epsilon(n)/2 \} }$, then $t_3 \geq 0$ is bounded from above by a constant only depending on $n$.

Again, we choose the upper half space model with $z = c_y(\infty)$ as the point $\infty$, and $y_1 = (0,1) \in \R^{n-1} \times \R_{>0}$ without restriction. Note that $\gamma$ acts on $\R^{n-1} \times \R_{>0} = \H^n$ as a Euclidean motion and on $\R_{>0}$ as the identity; we will denote the restricted action of $\gamma$ on $\R^{n-1}$ also by $\gamma$. By the triangle inequality we already have
\[
\epsilon' := d_{\gamma}((0,1)) \leq \frac{\epsilon(n)}{2} + 2 \cdot \frac{\epsilon(n)}{32} < \epsilon(n).
\]
With the usual distance formulas we deduce
\[
\epsilon' = d\big( (0,1), (\gamma(0),1) \big) = 2 \cdot \arsinh\left( \frac{\| 0 - \gamma(0) \|}{2 \cdot 1} \right) = 2 \cdot \arsinh\left( \frac{\| \gamma(0) \|}{2} \right),
\]
hence $\| \gamma(0) \| = 2 \sinh ( \epsilon'/2 )$. Since $c_y$ is parametrized by arc length, we get $c_y(t) = (0, e^t)$; so similar to the above, the displacement at time $t$ is $d_{\gamma}(c_y(t)) = d_{\gamma}((0,e^t)) = 2 \arsinh ( \| \gamma(0) \|/(2 e^t) )$. Using the above value for $\| \gamma(0) \|$ and the bound on $\epsilon'$, we obtain
\[
d_{\gamma}(c_y(t)) = 2 \cdot \arsinh\left( \frac{\sinh \left( \frac{\epsilon'}{2} \right)}{e^t} \right) < 2 \cdot \arsinh\left( \frac{\sinh \left( \frac{\epsilon(n)}{2} \right)}{e^t} \right).
\]
Hence if $t$ is so large that the term on the right hand side is at most $\epsilon(n)/2$, then $c_y(t)$ already lies in $\{ d_{\gamma} < \epsilon(n)/2 \}$; thus we have to solve
\[
2 \cdot \arsinh\left( \frac{\sinh \left( \frac{\epsilon(n)}{2} \right)}{e^t} \right) \stackrel{!}{\leq} \frac{\epsilon(n)}{2}
\]
for $t$, i.e.
\[
\ln \left( \frac{\sinh \left( \frac{\epsilon(n)}{2} \right)}{\sinh \left( \frac{\epsilon(n)}{4} \right)} \right) \leq t.
\]
So the left hand term tells us the maximal amount of time needed for $c_y$ to enter $\{ d_{\gamma} < \epsilon(n)/2 \}$; as it only depends on $n$, this finishes the proof.
\end{proof}

\subsubsection*{Properties of the covering}

We will now prove the desired properties of the covering.

\begin{lem}\label{Lem:Ueberdeckung_ist_stabil}
The sets of $\mathcal{B}$ form a covering of $M'_+$, i.e.
\[
M'_+ \subseteq \bigcup_{B \in \mathcal{B}} B,
\]
which is stable under the flow $f$ in the sense that if a flow geodesic enters $\bigcup_{B \in \mathcal{B}} B$, then it remains inside $\bigcup_{B \in \mathcal{B}} B$ until it meets $\partial M'_+$.
\end{lem}

\begin{proof}
Using Lemma \ref{Lem:Ueberdeckung_ueberdeckt}, $\bigcup_{B \in \mathcal{B}'} B \subseteq \bigcup_{B \in \mathcal{B}} B$ yields the first statement.

As every flow geodesic meets $\partial M'_+$ at some point, the stability is equivalent to the following property: a flow geodesic of a point $x \in \partial M'_+$ -- flowing to the thin part, i.e. in opposite direction -- is contained in $\bigcup_{B \in \mathcal{B}} B$ until it leaves this set at some point and never enters it again.

Let $x \in \partial M'_+$ and $i \in \{ 0, \ldots, n \}$ be minimal such that
\[
x \in \big( \partial M'_+ \cap (\pi(S_i))_{2\mu_i} \big) \setminus \bigcup_{j<i} (\pi(S_j))_{3\mu_j/2}.
\]
This is always possible: if $i=n$, then by $S_n = M$ we get $\partial M'_+ \cap (\pi(S_n))_{2\mu_n} = \partial M'_+$, i.e. $x$ lies in $\partial M'_+ \setminus \bigcup_{j<n} (\pi(S_j))_{3\mu_j/2}$ -- proving this statement --, or $x \in \partial M'_+ \cap \bigcup_{j<n} (\pi(S_j))_{3\mu_j/2}$. In the latter case, we can chose $j<n$ minimally such that $x \in (\pi(S_j))_{3\mu_j/2}$. Since $(\pi(S_j))_{2\mu_j} \supseteq (\pi(S_j))_{3\mu_j/2}$, this means $x \in \partial M'_+ \cap (\pi(S_j))_{2\mu_j}$, hence $x \notin \bigcup_{j'<j} (\pi(S_{j'}))_{3\mu_{j'}/2}$ by minimality of $j$. So
\[
x \in \big( \partial M'_+ \cap (\pi(S_j))_{2\mu_j} \big) \setminus \bigcup_{j'<j} (\pi(S_{j'}))_{3\mu_{j'}/2}.
\]
Thus we can always find a minimal $i$ as described above.

Let $c_x$ be the flow geodesic of $x$ in direction of $M_-$, where $c_x(0) = x$. By the above arguments, there is $x' \in \pi(S_i)$ with $d(x,x') < 2\mu_i$. Denote the corresponding flow geodesic (again in direction of $M_-$) of $x'$ by $c_{x'}$, where $c_{x'}(0) = x'$. Then $d(c_x(t), c_{x'}(t)) < 2\mu_i$ for $t>0$, because the distance between the geodesics decreases when flowing towards $M_-$. Note that $c_{x'}$ is entirely contained in $\pi(S_i)$, so $c_x(t) \in (\pi(S_i))_{2\mu_i}$ for $t>0$. As $d(c_{x'}(0), M'_+) = d(x', M'_+) \leq d(x',x) < 2\mu_i$, we see that $c_{x'}(t)$ has not entered $\partial \overline{(M'_+)_{8\mu_i}}$ for sufficiently small $t>0$; moreover, $c_x(t)$ does not lie in $(\pi(S_j))_{3\mu_j/2}$ for sufficiently small $t>0$, as (similar to the above arguments) this would contradict the minimality of $i$. With increasing $t$ we will now flow towards $M_-$ (equivalently: towards $\partial M_+$). If $c_{x'}(t)$ meets the boundary $\partial (M'_+)_{8\mu_i}$, we say that event I happened; if on the other hand $c_x(t)$ enters some $(\pi(S_j))_{3\mu_j/2}$ for $j<i$, we say that event II happened.

As $\partial M'_+ \subseteq M'_+$, $x$ is already contained in some $B \in \mathcal{B}$. Thus for sufficiently small $t>0$, we know that $c_x(t)$ lies inside $\bigcup_{B \in \mathcal{B}} B$ and neither event I nor II has happened.

We will now show that for increasing $t$, $c_x(t)$ will remain inside $\bigcup_{B \in \mathcal{B}} B$ if neither of these events happen. To this end note that $c_{x'}(t) \notin \bigcup_{j<i} (\pi(S_j))_{\mu_j}$ has to hold, because otherwise -- for some $j<i$ -- we would find $x'' \in \pi(S_j)$ with $d(x'', c_{x'}(t)) < \mu_j$. By 
\[
d(c_x(t), \pi(S_j)) \leq d(c_x(t), x'') \leq d(c_x(t), c_{x'}(t)) + d(c_{x'}(t), x'') < 2\mu_i + \mu_j < \frac{3}{2}\mu_j
\]
(recall $2\mu_i \leq \mu_j/6$, as $j<i$) this would mean $c_x(t) \in (\pi(S_j))_{3\mu_j/2}$, contradicting the assumption that event II hasn't happened. Since we further assumed that event I hasn't happened -- i.e. $c_{x'}(t)$ still lies inside $(M'_+)_{8\mu_i}$ --, we conclude
\[
c_{x'}(t) \in \big( \overline{(M'_+)_{8\mu_i}} \cap \pi(S_i) \big) \setminus \bigcup_{j<i} (\pi(S_j))_{\mu_j}.
\]
By definition, this set contains $\mathcal{D}_i$ as a maximal $\mu_i$-discrete subset, so there is $y \in \mathcal{D}_i$ with $d(y, c_{x'}(t)) < \mu_i$. This yields
\[
d(c_x(t), y) \leq d(c_x(t), c_{x'}(t)) + d(c_{x'}(t), y) < 2\mu_i + \mu_i = 3\mu_i,
\]
i.e. $c_x(t) \in B_{3\mu_i}^M(y) \subseteq \bigcup_{B \in \mathcal{B}} B$, what we wanted to show.

So it only remains to check what occurs if event I or event II happen. Let us begin with the case that event II happens first. Hence we can chose $j<i$ minimally such that $c_x(t) \in (\pi(S_j))_{3\mu_j/2} \subseteq (\pi(S_j))_{2\mu_j}$, where -- by minimality of $j$ -- also $c_x(t) \notin \bigcup_{j'<j} (\pi(S_{j'}))_{3\mu_{j'}/2}$. As event II happened before event I, we get
\[
d(c_x(t), M'_+) \leq d(c_x(t), c_{x'}(t)) + d(c_{x'}(t), M'_+) < 2\mu_i + 8\mu_i < \mu_j < 8 \mu_j
\]
(recall $j<i$), so
\[
c_x(t) \in \big( \overline{(M'_+)_{8\mu_j}} \cap (\pi(S_j))_{2\mu_j} \big) \setminus \bigcup_{j'<j} (\pi(S_{j'}))_{3\mu_{j'}/2}.
\]
Similar to the definition of $c_{x'}$, we can find a flow geodesic $c_{x''}$ which lies inside $\pi(S_j)$ and which fulfills $d(c_{x''}(t), c_x(t)) < 2\mu_j$ after event II happened. As
\[
d(c_{x''}(t), M'_+) \leq d(c_{x''}(t), c_x(t)) + d(c_x(t), M'_+) < 2\mu_j + \mu_j < 8\mu_j,
\]
the analogously defined event I for this index $j$ and the geodesics $c_x$, $c_{x''}$ hasn't happened yet; by minimality of $j$, the analogously defined event II for this data has also not happened. So replacing $i$ by $j$ and $c_{x'}$ by $c_{x''}$ in the above paragraphs, we can repeat the corresponding arguments and deduce that $c_x(t)$ still lies inside $\bigcup_{B \in \mathcal{B}} B$.

Let us now assume that event I happens first; hence $c_{x'}(t) \in \partial (M'_+)_{8\mu_i}$. Again, $c_{x'}(t) \notin \bigcup_{j<i} (\pi(S_j))_{\mu_j}$, because
\[
d(c_x(t), \pi(S_j)) \leq d(c_x(t), c_{x'}(t)) + d(c_{x'}(t), \pi(S_j)) < 2\mu_i + \mu_j < \frac{3}{2}\mu_j
\]
(for suitable $j<i$) would once more mean that event II has already happened, contradicting our assumption. Hence
\[
c_{x'}(t) \in \big( \partial (M'_+)_{8\mu_i} \cap \pi(S_i) \big) \setminus \bigcup_{j<i} (\pi(S_j))_{\mu_j}.
\]
By definition, $\mathcal{D}_i$ contains a maximal $\mu_i$-discrete subset inside this set, so there is some $y \in \mathcal{D}_i$ with $d(y, c_{x'}(t)) < \mu_i$. Again, $d(c_x(t), y) < 3\mu_i$, thus $c_x(t) \in B_{3\mu_i}^M(y) \subseteq \bigcup_{B \in \mathcal{B}} B$, what we wanted to show. Note that by our choice of stretched balls, $B_{3\mu_i}^M(y)$ is the initial ball of some stretched ball $U$; consequently, $c_x(t)$ lies inside $U$ from this point on. As seen in the proof of Lemma \ref{Lem:Laenge_Streckung}, $c_x$ will remain outside of $\bigcup_{B \in \mathcal{B}} B$ after leaving $U$, or lie inside some other stretched ball $U'$; in the latter case, we can repeat the argument, proving the statement.
\end{proof}

The following lemmata state that the covering sets and their intersections are contractible.

\begin{lem}\label{Lem:Ueberdeckung_Kugeln_zusammenziehbar}
Elements of $\mathcal{B}$ are folded sets and thus contractible.
\end{lem}

\begin{proof}
The proof is similar to \cite{Samet} Theorem 4.2 step 2). Let $x \in \mathcal{D}_i$ and $\widetilde{x} \in X$ be a preimage of $x$ in $X$; moreover, let $Y \in \Sigma_i$ be the singular submanifold containing $\widetilde{x}$. By choice of $\mathcal{D}_i$ we have $d(x, \pi(S_{<i})) > \mu_{i-1}$ and thus $d(\widetilde{x}, S_{<i}) > \mu_{i-1}$. Note that since $x \in \mathcal{D}_i \subseteq M_+$, also $\widetilde{x} \in X_+$, hence $\widetilde{x} \in Y \cap X_+$.

By the definition of the $\mu_i$ we have $24\mu_i \leq \epsilon_3(\mu_{i-1})$, so using Lemma \ref{Samet4.7} we conclude that $\Gamma_{24\mu_i}(\widetilde{x})$ fixes $Y$ pointwise. Hence the assumptions of Lemma \ref{Lem:Durchschnitt_Kugeln_faltbar} and \ref{Lem:Parabolisch_gestreckte_Kugeln_faltbar} are fulfilled; using these, we see that the respective ball with (initial) center $\widetilde{x}$ is $Y$-foldable, which proves the statement.
\end{proof}

\begin{lem}\label{Lem:Ueberdeckung_Durchschnitte_zusammenziehbar}
Nonempty intersections of the elements of $\mathcal{B}$ are folded sets and thus contractible.
\end{lem}

\begin{proof}
The main idea is similar to \cite{Samet} Theorem 4.2 step 3), although we have to work considerably harder because of the presence of stretched balls.

Let $x_1, \ldots, x_k \in \mathcal{D}$, where $x_j \in \mathcal{D}_{n_j}$ for $j=1,\ldots,k$, such that the corresponding sets $U_j$ of $\mathcal{B}$ have a non-empty intersection. Note that these balls are either all ordinary balls, or we have an intersection of stretched balls and possibly some ordinary balls. The proof for the case of ordinary balls only is essentially the same as in \cite{Samet} Theorem 4.2 step 3), so we will skip it. Without restriction, let $n_1$ be the maximum of the $n_j$ ($j=1,\ldots,k$).

In a first step, we will treat the case where all the balls are stretched; so let $B_{3\mu_{n_j}}^X(\widetilde{x}_j)$ be the corresponding initial balls of the $U_j$ and note that $\widetilde{x}_j \in X_+$ for all $j=1,\ldots,k$. Denote the singular submanifold containing $\widetilde{x}_j$ by $Y_j \in \Sigma_{n_j}$. By definition of $\mathcal{D}_j$, we have $d(\widetilde{x}_j, S_{<n_j}) \geq \mu_{n_j - 1}$. Observe that by maximality of $n_1$, the point $\widetilde{x}_1$ is the farthest away from the thin part among all $\widetilde{x}_j$: for $j \in \{1,\ldots,k\}$ with $n_1 > n_j$ (i.e. $\mu_{n_j} > \mu_{n_1}$) this follows from Lemma \ref{Lem:Monotonie_der_Radien}; and for those $j$ with $n_1 = n_j$ (i.e. $\mu_{n_1} = \mu_{n_j}$) the point $\widetilde{x}_1$ can be chosen to be the one maximizing the distance to the thin part (among all $\widetilde{x}_j$ with $n_j = n_1$), without restriction. Let $U_1^{(j)}$ be the comparison ball of $U_1$ at the height of $\widetilde{x}_j$, with initial center $\widetilde{x}_1^{(j)}$ and initial radius $3\mu_{n_1}^{(j)}$. Note that $\widetilde{x}_1^{(j)}$ lies in the same singular submanifolds as $\widetilde{x}_1$, so in particular $\widetilde{x}_1^{(j)} \in Y_1$. Since $\widetilde{x}_1$ had the largest distance to the thin part, we get $3\mu_{n_1}^{(j)} \leq 3\mu_{n_1}$ for all $j=2,\ldots,k$; by maximality of $n_1$, we moreover have $\mu_{n_1} \leq \mu_{n_j}$ for all $j$, hence $3\mu_{n_1}^{(j)} \leq 3\mu_{n_j}$. The construction of the comparison balls now allow the application of Lemma \ref{Lem:Parabolisch_gestreckte_Kugeln_Abstand}, which yields
\[
d(\widetilde{x}_j, \widetilde{x}_1^{(j)}) < 2\cdot 3\mu_{n_j} + 2 \cdot 3\mu_{n_1}^{(j)} \leq 12\mu_{n_j}.
\]
Since $12\mu_{n_j} \leq \epsilon_2(\mu_{n_j - 1})$ and $d(\widetilde{x}_j, S_{<n_j}) \geq \mu_{n_j - 1}$, we can use Lemma \ref{Samet4.6} and deduce $Y_j \subseteq Y_1$. Hence all initial centers lie in the same singular submanifold $Y_1$, which already contained the center of the $Y_1$-foldable stretched ball $U_1$. As $\Gamma_{24\mu_{n_j}}(\widetilde{x}_j)$ (for all $j=1,\ldots,k$) fixes $\widetilde{x}_j$ (see the proof of the previous Lemma \ref{Lem:Ueberdeckung_Kugeln_zusammenziehbar}), the statement follows after applying Lemma \ref{Lem:Durchschnitt_parabolisch_gestreckte_Kugeln_faltbar}.

It remains to examine what happens if we have a mixed intersection of ordinary and stretched balls. This situation can be reduced to the case of the intersection of a single ordinary ball and a single stretched ball: if $U_1, \ldots, U_k$ are stretched balls as above such that $n_1 \geq n_j$ for all $j=1,\ldots,k$, we saw that the initial centers of the $U_j$ also lie in the singular submanifold $Y_1$, and a similar statement holds for ordinary balls $U'_1, \ldots, U'_{k'}$ (so the centers of the $U'_j$ lie in $Y'_1$, which contained the center of $U'_1$); so if we can show that either $Y_1 \subseteq Y'_1$ or $Y'_1 \subseteq Y_1$, then all (initial) centers of the mixed intersection would be contained in the same singular submanifold. Applying Lemma \ref{Lem:Durchschnitt_parabolisch_gestreckte_Kugeln_faltbar}, this would yield the statement (the additional assumption that $\Gamma_{24\mu_{n_j}}(\widetilde{x}_j)$ fixes the centers $\widetilde{x}_j$ can again be taken from the proof of Lemma \ref{Lem:Ueberdeckung_Kugeln_zusammenziehbar}).

So let $U_1$ be a stretched ball with initial ball $B_{3\mu_{n_1}}^X(\widetilde{x}_1)$ and $U_2 = B_{3\mu_{n_2}}^X(\widetilde{x}_2)$ be an ordinary ball; denote the singular submanifolds containing $\widetilde{x}_1$ and $\widetilde{x}_2$ by $Y_1 \in \Sigma_{n_1}$ and $Y_2 \in \Sigma_{n_2}$, respectively. We want to show that one of $Y_1 \subseteq Y_2$ or $Y_2 \subseteq Y_1$ always holds.

By the choice of centers, $U_1$ can not intersect an ordinary ball of strictly smaller radius (compare the proof of Lemma \ref{Lem:Monotonie_der_Radien}), so we already have $3\mu_{n_1} \leq 3\mu_{n_2}$, i.e. $n_1 \geq n_2$. If the initial ball of $U_1$ intersects (the ordinary ball) $U_2$, then the above argument for the intersection of ordinary balls can be used, giving $Y_2 \subseteq Y_1$, which proves the statement. Thus without restriction, we can assume that the intersection of $U_1$ and $U_2$ is outside the initial ball of $U_1$. We will distinguish between two cases, depending on the proximity of $\widetilde{x}_1$ and $\widetilde{x}_2$ to the thin part.

\begin{itemize}
\item In the first case, $\widetilde{x}_1$ is closer to the thin part than $\widetilde{x}_2$. In a first step, assume $n_1 = n_2$, i.e. $3\mu_{n_1} = 3\mu_{n_2}$. Note that since $U_1$ and $U_2$ intersect, $U_1$ also intersects the stretched ball $U'_2$ with initial ball $U_2$; let $U''_2$ be the comparison ball of $U'_2$ at height $\widetilde{x}_1$. As $\widetilde{x}_1$ is closer to the thin part than $\widetilde{x}_2$, the initial radius $3\mu''_{n_2}$ of $U''_2$ satisfies $3\mu''_{n_2} \leq 3\mu_{n_2} (=3\mu_{n_1})$. If $\widetilde{x}''_2$ denotes the initial center of $U''_2$ (observe that $\widetilde{x}''_2 \in Y_2$), then using Lemma \ref{Lem:Parabolisch_gestreckte_Kugeln_Abstand} we deduce
\[
d(\widetilde{x}_1, \widetilde{x}''_2) < 2 \cdot 3\mu_{n_1} + 2 \cdot 3\mu''_{n_2} \leq 12\mu_{n_1}.
\]
By definition, $12\mu_{n_1} \leq \epsilon_2(\mu_{n_1 - 1})$ and $d(\widetilde{x}_1, S_{<n_1}) \geq \mu_{n_1 - 1}$, so Lemma \ref{Samet4.6} yields $Y_2 \subseteq Y_1$, the statement.

Assume now that $n_1 > n_2$, i.e. $3\mu_{n_1} < 3\mu_{n_2}$. Let $x \in U_1 \cap U_2$ and $c_x$ be the flow geodesic towards the parabolic fixed point $z$, with parametrization $c_x(0) = x$ and $c_x(-\infty) = z$; moreover, let $HS_1$ and $HS_2$ be the horospheres around $z$ containing $\widetilde{x}_1$ and $\widetilde{x}_2$, respectively. As $\widetilde{x}_1$ is closer to the thin part and the initial ball of $U_1$ does not intersect $U_2$ -- i.e. $c_x$ (coming from $z$) leaves $U_2$ before entering the initial ball of $U_1$ --, we have $c_x(t_1) \in HS_1$ and $c_x(t_2) \in HS_2$ for suitable $0 < t_1 < t_2$, where $t_2 < 3\mu_{n_2}$ (recall that $x \in U_2 = B_{3\mu_{n_2}}^X(\widetilde{x}_2)$). Note that -- on its way from $x$ to $c_x(t_1)$ -- $c_x$ has to meet the initial ball $B_{3\mu_{n_1}}^X(\widetilde{x}_1)$ of $U_1$ (compare the proof of Lemma \ref{Lem:Monotonie_der_Radien}). Thus there is some $t'$ with $0 < t' < t_1$ (hence $t' < 3\mu_{n_2}$) such that $c_x(t') \in B_{3\mu_{n_1}}^X(\widetilde{x}_1)$. We deduce
\begin{align*}
d(\widetilde{x}_1, \widetilde{x}_2) &\leq d(\widetilde{x}_1, c_x(t')) + d(c_x(t'), c_x(0)) + d(c_x(0), \widetilde{x}_2) \\
&= d(\widetilde{x}_1, c_x(t')) + t' + d(x, \widetilde{x}_2) \\
&< 3\mu_{n_1} + t' + 3\mu_{n_2} \\
&< 3\mu_{n_1} + 3\mu_{n_2} + 3\mu_{n_2} \\
&< 7\mu_{n_2}.
\end{align*}
Since $7\mu_{n_2} < 12\mu_{n_2} \leq \epsilon_2(\mu_{n_2 - 1})$, we conclude $d(\widetilde{x}_1, \widetilde{x}_2) < \epsilon_2(\mu_{n_2 - 1})$. By definition of $\mathcal{D}_i$, we also have $d(\widetilde{x}_2, S_{<n_2}) \geq \mu_{n_2 - 1}$; applying Lemma \ref{Samet4.6} gives $Y_2 \subseteq Y_1$, the desired statement.

\item In the other case we assume that $\widetilde{x}_1$ is farther away from the thin part than $\widetilde{x}_2$. So the comparison ball $U'_1$ of $U_1$ at the height of $\widetilde{x}_2$ satisfies $3\mu'_{n_1} \leq 3\mu_{n_1}$, where $B_{3\mu'_{n_1}}^X(\widetilde{x}'_1)$ denotes  the initial ball of $U'_1$ (note that $\widetilde{x}'_1 \in Y_1$, as above). Using Lemma \ref{Lem:Parabolisch_gestreckte_Kugeln_Abstand} and $3\mu_{n_1} \leq 3\mu_{n_2}$, we deduce
\[
d(\widetilde{x}'_1, \widetilde{x}_2) < 2 \cdot 3\mu'_{n_1} + 2 \cdot 3\mu_{n_2} \leq 12\mu_{n_2}.
\]
Since $12\mu_{n_2} \leq \epsilon_2(\mu_{n_2 - 1})$ und $d(\widetilde{x}_2, S_{<n_2}) \geq \mu_{n_2 - 1}$, applying Lemma \ref{Samet4.6} again yields $Y_2 \subseteq Y_1$, which finishes the proof.
\end{itemize}
\end{proof}

Our construction guarantees that we can control the number of covering sets and the number of nonempty intersections between such sets linearly by the volume:

\begin{lem}\label{Lem:Ueberdeckung_Anzahl_und_Schnitte}
There are constants $C = C(n, \eta, \nu), D = D(n, \nu) > 0$ satisfying the following statements.
\begin{enumerate}
\item We have $|\mathcal{B}| \leq C \cdot \Vol(M)$.

\item A set of $\mathcal{B}$ intersects at most $D$ other sets of $\mathcal{B}$.
\end{enumerate}
\end{lem}

\begin{proof}
The proof of the first statement is a standard argument which is basically identical to \cite{Samet} Theorem 4.2 step 4). Note that our $\eta$ corresponds to the $m$ in \cite{Samet}, which explains the dependency of $C$ on $\eta$; the dependency of $C$ on $\nu$ stems from the fact that we defined our thick part using $\nu$ as a lower bound on the hyperbolic displacement (so the thin part consisted only of cusps).

Note that in order to estimate the maximal number of intersecting sets, we can replace the stretched balls by larger ordinary balls is in Lemma \ref{Lem:Laenge_Streckung_absolut} and use the resulting value as an upper bound. But for the (larger) ordinary balls, the argument is essentially the same as in \cite{Samet} Theorem 4.2, this time step 5)\footnote{The usual proof shows that there is no dependency of $D$ on $\eta$; the dependency on $\nu$ again comes from the fact that we needed $\nu$ to put the tubes in the thick part.}.

Detailed proofs of these statements not omitting these details can also be taken from \cite{SenskaDissertation} Lemma 3.50.
\end{proof}

\subsection{Nerve construction}

Our goal is now to build the desired simplicial complex out of the covering via the nerve construction. Let $\mathcal{B}_g \subseteq \mathcal{B}$ be the set of all stretched balls in $\mathcal{B}$; define
\[
N_+ := \bigcup_{B \in \mathcal{B}} B \qquad\text{and}\qquad N_0 := \bigcup_{B \in \mathcal{B}_g} B.
\]
We denote the nerve complexes (see also \cite{Sauer} section 2.2) associated to $\mathcal{B}$ and $\mathcal{B}_g$ by $N(\mathcal{B})$ and $N(\mathcal{B}_g)$, respectively; obviously, we can think of $N(\mathcal{B}_g)$ as a subcomplex of $N(\mathcal{B})$. The following lemma will be needed to prove the homotopy equivalence between the nerve complexes and the thick part.

\begin{lem}\label{Lem:Kommutatives_Diagramm_M_+_N_+}
There is a homotopy equivalence $F: M_+ \stackrel{\simeq}{\rightarrow} N_+$ which induces a commutative diagram, with vertical maps given by the inclusions $\partial M_+ \hookrightarrow M_+$ and $N_0 \hookrightarrow N_+$:
\begin{center}
\begin{tikzcd}
\partial M_+ \ar[r, "\simeq", "F|_{\partial M_+}"']	\ar[d, hook]		& N_0 \ar[d, hook] \\
M_+ \ar[r, "\simeq", "F"']																					& N_+.
\end{tikzcd}
\end{center}
\end{lem}

\begin{proof}
Let $M''_+$ be the shrinking of the thick part $M_+$ by $3\mu_{-1}/4$, i.e. $M''_+ := M \setminus (M_-)_{{3\mu_{-1}}/{4}}$. By the choice of the stretching length (see Lemma \ref{Lem:Laenge_Streckung}) we know that every stretched ball $B \in \mathcal{B}_g$ intersects $\partial M''_+$; on the other hand, no ordinary ball of $\mathcal{B}$ meets the boundary $\partial M''_+$, see Figure \ref{Bild_Beweis_Homotopieaequivalenz_M+_N+_Lage_der_Mengen}.

{
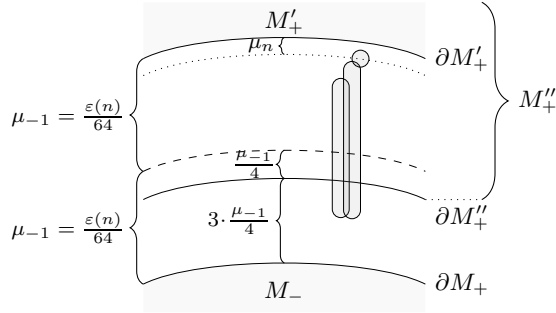
\begin{figure}
\centering

\begin{tikzpicture}[scale=0.75]


\begin{scope}[xshift=0cm,yshift=0cm]


\fill[color=gray, opacity=0.05] (0,0) .. controls (1,0.5) and (4,0.5) .. (5,0) -- (5,-0.5) -- (0,-0.5) -- cycle;
\draw[very thin] (0,0) .. controls (1,0.5) and (4,0.5) .. (5,0) node[right] {\footnotesize$\partial M_+$};
\node at (2.5,-0.15) {\footnotesize$M_-$};


\draw[very thin, dashed] (0,2) .. controls (1,2.5) and (4,2.5) .. (5,2);


\fill[color=gray, opacity=0.05] (0,4) .. controls (1,4.5) and (4,4.5) .. (5,4) -- (5,5) -- (0,5) -- cycle;
\draw[very thin] (0,4) .. controls (1,4.5) and (4,4.5) .. (5,4) node[right] {\footnotesize$\partial M'_+$};
\node at (2.5,4.65) {\footnotesize$M'_+$};


\draw[very thin] (0,1.5) .. controls (1,2) and (4,2) .. (5,1.5) node[right,yshift=-7pt] {\footnotesize$\partial M''_+$};

\draw[thin, dotted] (5,1.5) -- (6,1.5);

\draw [decorate,decoration={brace,mirror,amplitude=10pt},xshift=1cm,yshift=0pt]
(5,1.5) -- (5,5) node [black,midway,xshift=0.75cm] 
{\footnotesize $M''_+$};


\draw[thin, dotted] (0,3.7) .. controls (1,4.2) and (4,4.2) .. (5,3.7);


\draw [decorate,decoration={brace,amplitude=5pt},xshift=0pt,yshift=0pt]
(0,0) -- (0,2) node [black,midway,xshift=-1cm] 
{\scriptsize $\mu_{-1} = \frac{\epsilon(n)}{64}$};

\draw [decorate,decoration={brace,amplitude=5pt},xshift=0pt,yshift=0pt]
(0,2) -- (0,4) node [black,midway,xshift=-1cm] 
{\scriptsize $\mu_{-1} = \frac{\epsilon(n)}{64}$};


\draw[decorate,decoration={brace,amplitude=3pt},xshift=0pt,yshift=0pt]
(2.5,0.375) -- (2.5,1.875) node [black,midway,xshift=-0.6cm] 
{\scriptsize $3\!\cdot\!\frac{\mu_{-1}}{4}$};

\draw[decorate,decoration={brace,amplitude=3pt},xshift=0pt,yshift=0pt]
(2.5,1.875) -- (2.5,2.375) node [black,midway,xshift=-0.4cm] 
{\scriptsize $\frac{\mu_{-1}}{4}$};

\draw[decorate,decoration={brace,amplitude=2pt},xshift=0pt,yshift=0pt]
(2.5,4.075) -- (2.5,4.375) node [black,midway,xshift=-0.3cm] 
{\scriptsize$\mu_n$};



\fill[color=gray, opacity=0.1] (3.65,3.5) arc (0:180:0.15cm) -- (3.35,1.33) arc (180:360:0.15cm) -- cycle;
\draw[ultra thin] (3.65,3.5) arc (0:180:0.15cm) -- (3.35,1.33) arc (180:360:0.15cm) -- cycle;


\fill[color=gray, opacity=0.1] (3.85,3.8) arc (0:180:0.15cm) -- (3.55,1.31) arc (180:360:0.15cm) -- cycle;
\draw[ultra thin] (3.85,3.8) arc (0:180:0.15cm) -- (3.55,1.31) arc (180:360:0.15cm) -- cycle;


\fill[color=gray, opacity=0.1] (3.85,4) circle (0.15cm);		
\draw[ultra thin] (3.85,4) circle(0.15cm);									


\end{scope}


\end{tikzpicture}

\caption{Simplified depiction of the position of the covering sets. Ordinary balls do not exceed the dashed line in the middle and thus can not intersect $\partial M''_+$. Stretched balls always intersect $\partial M''_+$, but do not meet the upper dotted line. $N_+$ consists of all balls (ordinary and stretched), whereas $N_0 \subseteq N_+$ is made up of all stretched balls. Observe that $M''_+ \subseteq N_+$.
}
\label{Bild_Beweis_Homotopieaequivalenz_M+_N+_Lage_der_Mengen}

\end{figure}
}

Similar to Lemma \ref{Lem:Fluss_starke_Deformationsretraktion}, the flow away from the thin part up to $\partial M''_+$ (where we stop flowing) induces a commutative diagram
\begin{center}
\begin{tikzcd}
\partial M_+ \ar[r, "\simeq"]	\ar[d, hook]		& \partial M''_+ \ar[d, hook] \\
M_+ \ar[r, "\simeq"]													& M''_+.
\end{tikzcd}
\end{center}
Moreover, $N_+$ is stable under this flow (compare Lemma \ref{Lem:Ueberdeckung_ist_stabil}); by the choice of the shrinking $3\mu_{-1}/4$, the same is true for $N_0$. Let $N'_+$ and $N'_0$ denote the images of $N_+$ and $N_0$ under this flow up to $\partial M''_+$, then we get a commutative diagram
\begin{center}
\begin{tikzcd}
N_0 \ar[r, "\simeq"]	\ar[d, hook]		& N'_0 \ar[d, hook] \\
N_+ \ar[r, "\simeq"]									& N'_+.
\end{tikzcd}
\end{center}
Observe that $N'_+ = M''_+$. On the other hand -- by the construction of the stretched balls --, $N'_+$ and $N'_0$ are also stable under the flow in opposite direction (i.e. flowing towards $M_-$) up to $\partial M''_+$; this yields a homotopy equivalence $N'_+ \simeq M''_+$, which itself induces a homotopy equivalence $N'_0 \simeq \partial M''_+$. Consequently, we get the commutative diagram
\begin{center}
\begin{tikzcd}
N'_0 \ar[r, "\simeq"]	\ar[d, hook]		& \partial M''_+ \ar[d, hook] \\
N'_+ \ar[r, "\simeq"]									& M''_+.
\end{tikzcd}
\end{center}
The desired diagram of the statement is now obtained by composing the above diagrams, where -- if needed -- the horizontal arrows can be reversed by taking the respective homotopy inverses.
\end{proof}

Our main result will depend on the following homotopy equivalence.

\begin{lem}\label{Lem:Dicker_Teil_homotopieaequivalent_zu_Simplizialkomplex}
$(M_+, \partial M_+)$ is as a pair homotopy equivalent to $(N(\mathcal{B}), N(\mathcal{B}_g))$.
\end{lem}

\begin{proof}
By Lemma \ref{Lem:Ueberdeckung_Kugeln_zusammenziehbar} and \ref{Lem:Ueberdeckung_Durchschnitte_zusammenziehbar} we know that the (open) covers of $N_+$ and $N_0$ by $\mathcal{B}$ and $\mathcal{B}_g$, respectively, are good covers\footnote{In the sense that the covering sets and their non-empty intersections are contractible.}. Using \cite{Sauer} Theorem 2.7 we get a commutative diagram
\begin{center}
\begin{tikzcd}
N_0 \ar[d, hook, "j"] 		& (N_0)_{\mathcal{B}_g} \ar[d, hook, "k"] \ar[l, "\simeq"'] \ar[r, "\simeq"]		& N(\mathcal{B}_g) \ar[d, hook] \\
N_+							 		& (N_+)_{\mathcal{B}} \ar[l, "\simeq"', "G"] \ar[r, "\simeq"]										& N(\mathcal{B}),
\end{tikzcd}
\end{center}
where the vertical maps are given by the respective inclusions; here, the spaces $(N_0)_{\mathcal{B}_g}$ and $(N_+)_{\mathcal{B}}$ are the generalized nerve spaces as in \cite{Sauer} section 2.2. Attaching the diagram of Lemma \ref{Lem:Kommutatives_Diagramm_M_+_N_+} on the left hand side yields the commutative diagram
\begin{equation}
\begin{tikzcd}\label{Diagramm:Beweis_Hauptresultat_Orbifaltigkeiten}\tag{$*$}
\partial M_+ \ar[d, hook] \ar[r, "\simeq"]		& N_0 \ar[d, hook, "j"] 		& (N_0)_{\mathcal{B}_g} \ar[d, hook, "k"] \ar[l, "\simeq"'] \ar[r, "\simeq"]		& N(\mathcal{B}_g) \ar[d, hook] \\
M_+ \ar[r, "\simeq", "F"']													& N_+							 		& (N_+)_{\mathcal{B}} \ar[l, "\simeq"', "G"] \ar[r, "\simeq"]										& N(\mathcal{B}),
\end{tikzcd}
\end{equation}
where the outer vertical maps are cofibrations, because $\partial M_+ \hookrightarrow M_+$ is the inclusion of the boundary and $N(\mathcal{B}_g) \hookrightarrow N(\mathcal{B})$ the inclusion of a subcomplex.

The diagram (\ref{Diagramm:Beweis_Hauptresultat_Orbifaltigkeiten}) can take the role of diagram (11) in step 4 of the proof of \cite{Sauer} Theorem 3.1. Note that the second column of that diagram (11) has no counterpart in our diagram (\ref{Diagramm:Beweis_Hauptresultat_Orbifaltigkeiten}); this is no issue, as that column was only needed to construct a diagram similar to the one in our Lemma \ref{Lem:Kommutatives_Diagramm_M_+_N_+}. Repeating the arguments of \cite{Sauer} Theorem 3.1 step 4 with our diagram (\ref{Diagramm:Beweis_Hauptresultat_Orbifaltigkeiten}) replacing diagram (11) there now yields the statement.
\end{proof}

Of course the impact of Lemma \ref{Lem:Dicker_Teil_homotopieaequivalent_zu_Simplizialkomplex} depends on if we can control the complexity of the simplicial pair $(N(\mathcal{B}), N(\mathcal{B}_g))$; here, Lemma \ref{Lem:Ueberdeckung_Anzahl_und_Schnitte} will come into play. Before summarizing all these statements in our main result, we will see that the case of arithmetic, non-uniform lattices $\Gamma$ is particularly nice -- all constants will only depend on the dimension $n$ (and not on the other constants $\eta$ and $\nu$):

\begin{lem}\label{Lem:Arithmetisch_nicht-uniform_gute_Konstanten}
Let $\mathfrak{A}_n$ be the class of arithmetic, non-uniform lattices in $\Isom(\H^n)$, then:
\begin{enumerate}
\item There is a constant $\eta = \eta(n) \in \N$ only depending on $n$, such that for all $\Gamma \in \mathfrak{A}_n$, the order of any finite subgroup $G$ of $\Gamma$ is bounded by $\eta$, i.e. $|G| \leq \eta$.

\item There is a constant $\nu = \nu(n) > 0$ only depending on $n$, such that for all $\Gamma \in \mathfrak{A}_n$ and every hyperbolic $\gamma \in \Gamma$, the minimal displacement of $\gamma$ is at least $\nu$, i.e. $d_{\gamma}(x) \geq \nu$ for all $x \in \H^n$.
\end{enumerate}
\end{lem}

\begin{proof}
These statements are straightforward consequences of \cite{Gelander} Lemma 13.1 and \cite{Gelander} Remark 5.7.
\end{proof}

We will now state the main result in its general form; recall that a pair $(S,S')$ with a simplicial complex $S$ and (possibly empty) subcomplex $S' \subseteq S$ is an \textbf{$(A,B)$-simplicial pair}, if $S$ has at most $B$ vertices and the degree at every vertex is bounded by $A$.

\begin{satz}\label{Satz:Hauptresultat_beschraenkt_komplizierter_Simplizialkomplex}
As usual, let $\eta \in \N$ be an upper bound on the order of finite subgroups of $\Gamma$ and $\nu > 0$ be a lower bound on the displacement of hyperbolic elements of $\Gamma$. Then there are constants $C = C(n,\eta,\nu)$ and $D = D(n,\nu)$, such that $(M_+, \partial M_+)$ is as a pair homotopy equivalent to a $(D, C \cdot \Vol(M))$-simplicial pair. For $\Gamma \in \mathfrak{A}_n$ (i.e. $\Gamma$ is arithmetic, non-uniform) the constants $C$ and $D$ will only depend on the dimension $n$.
\end{satz}

\begin{proof}
This is a combination of Lemma \ref{Lem:Dicker_Teil_homotopieaequivalent_zu_Simplizialkomplex} and \ref{Lem:Ueberdeckung_Anzahl_und_Schnitte}, as well as \ref{Lem:Arithmetisch_nicht-uniform_gute_Konstanten} in the arithmetic, non-uniform case.
\end{proof}

\begin{bem}\label{Bem:Diskussion_Abhaengigkeit_Konstanten_von_eta_nu}
Using the formulas given for the construction of the constant $C$ of Theorem \ref{Satz:Hauptresultat_beschraenkt_komplizierter_Simplizialkomplex}, it can be shown that $C = C(n,\eta,\nu)$ grows exponentially in $\eta$. The dependency on $\nu$ can not be deduced that easily, as it was also used for the definition of $\mu_{-1}$, and thus its influence would have to traced along the iterative construction of all the $\mu_i$ up to $\mu_n$.
\end{bem}

\section{Applications}\label{Kapitel:Anwendungen}

Using the main result Theorem \ref{Satz:Hauptresultat_beschraenkt_komplizierter_Simplizialkomplex} (and its notation), bounds on the homology of hyperbolic orbifolds are an immediate consequence.

\begin{satz}\label{Satz:Freie_Homologie_Schranke}
Let $\K$ be an arbitrary field and let $b_k(M;\K) = \dim_{\K} H_k(M;\K)$ denote the $k$-th Betti number of $M$ with coefficients in $\K$. Then there is a constant $E = E(n, \eta, \nu) > 0$ such that
\[
b_k(M;\K) \leq E \cdot \Vol(M)
\]
for all $k\in\N_0$. In the arithmetic, non-uniform case, the constant $E$ will only depend on the dimension $n$.
\end{satz}

\begin{proof}
With our main result Theorem \ref{Satz:Hauptresultat_beschraenkt_komplizierter_Simplizialkomplex}, the proof is a standard argument utilizing the Mayer-Vietoris sequence (see e.g. \cite{SenskaVis} Theorem 4.11 or \cite{SenskaDissertation} Satz 3.59); $E$ will be given by $E := (D^{n-1} + D^n + 1) \cdot C$.
\end{proof}

The similar result for the torsion part of the homology is given as follows.

\begin{satz}\label{Satz:Torsion_Homologie_Schranke}
There is a constant $F = F(n, \eta, \nu) > 0$ such that
\[
\log | \tors H_k(M;\Z) | \leq F \cdot \Vol(M)
\]
for all $k\in\N_0$. In the arithmetic, non-uniform case, the constant $F$ will only depend on $n$.
\end{satz}

\begin{proof}
By the main result Theorem \ref{Satz:Hauptresultat_beschraenkt_komplizierter_Simplizialkomplex}, $M \simeq M_+$ is homotopy equivalent to a $(D, C \cdot \Vol(M))$-simplicial complex. Now \cite{Sauer} Lemma 5.2 yields the result. Note that -- following the proof of \cite{Sauer} Lemma 5.2 and assuming $D \geq 1$ (without restriction) -- the constant $F$ will be given by $F := D^n \cdot \log(n+2) \cdot C$.
\end{proof}

As we see from the previous two theorems, the constants for the homology bounds depend on the dimension $n$, the maximal order $\eta$ of finite subgroups of the lattice and the minimal hyperbolic displacement $\nu$. Using a more general approach than the one presented in this paper, the dependence on $\nu$ can be relaxed -- yielding bounds independent of $\nu$, but polynomial in $\Vol(M)$ -- and we suspect that it might be removed altogether. This more general approach basically consists of replicating the construction of the stretched balls also near the tubes (which we basically ignored by putting them into the thick part, using $\nu$); a detailed description of this procedure and the corresponding proofs are contained in \cite{SenskaDissertation} and might be a starting point for future improvements of the results given here.

\end{document}